\documentclass[a4paper]{article}
\usepackage[pages=all, color=black, position={current page.south}, placement=bottom, scale=1, opacity=1, vshift=5mm]{background}\SetBgContents{}
\usepackage[margin=1in]{geometry} 
\RequirePackage{fix-cm}
\usepackage{amsmath,amsfonts,amssymb,epsfig}
\usepackage{graphicx}
\usepackage{amssymb}
\usepackage{mathptmx}      
\usepackage{mathtools} 
\usepackage{algorithm}
\usepackage[noend]{algpseudocode}
\usepackage{faktor}

\usepackage[T1]{fontenc}
\usepackage[utf8]{inputenc}
\usepackage[english]{babel}
\usepackage{setspace}
\usepackage{amsthm}  
\usepackage{array}
\usepackage{float}
\usepackage{graphicx}
\usepackage{bm}
\usepackage{fancyhdr}
\usepackage{longtable}
\usepackage{epstopdf}
\usepackage{hyperref}
\usepackage{tikz}
\usepackage{pgfplots}
\usepackage{color}
\usepackage{float}
\usepackage{multicol}
\usepackage{subcaption}
\usepackage{latexsym}
\usepackage{hyperref}

\newtheorem{remark}{Remark}
\newtheorem{thm}{Theorem}
\newtheorem{cor}[thm]{Corollary}
\newtheorem{lem}[thm]{Lemma}

\newtheorem{prop}[thm]{Proposition}
\newtheorem{assumption}{Assumption}
\newtheorem{defn}{Definition}

\def\ItemNN$#1${\item $\displaystyle#1$}

\DeclareMathOperator*{\argmin}{argmin}

\DeclareMathOperator*{\osc}{osc}
\DeclareMathOperator*{\Span}{span}

\makeatletter
\def\BState{\State\hskip-\ALG@thistlm}
\makeatother

\usepackage{enumitem} 
\usepackage{marginnote}

\newcommand{\Vx}{\mathbf{x}}

\newcommand{\Vb}{\mathbf{b}}
\newcommand{\Vw}{\mathbf{w}}

\newcommand{\Vz}{\mathbf{z}}

\title{Robust Variational Physics-Informed Neural Networks}

\author{ Sergio Rojas $^a$     
        \and
        Pawe\l~Maczuga $^b$
        \and 
        Judit Mu\~noz-Matute$^{c,d}$
        \and
        David Pardo$^{e,c,f}$
        \and
        Maciej Paszy\'nski$^b$
}

\date{$^a$ Instituto de Matem\'aticas,  
              Pontificia Universidad Cat\'olica de Valpara\'iso, Chile. 
              \texttt{sergio.rojas.h@pucv.cl.}
             \\
           $^b$ AGH University of Krakow, Poland.\\
           $^c$ Basque Center for Applied Mathematics (BCAM), Spain. \\
           $^d$ Oden Institute for Computational Engineering and Sciences, The University of Texas at Austin, USA.\\
           $^e$ University of the Basque Country (UPV/EHU), Spain.\\
           $^f$ Ikerbasque, Spain.\\[2ex]  \today}
\begin{document}
\maketitle
\begin{abstract}
We introduce a Robust version of the Variational Physics-Informed Neural Networks method (RVPINNs). As in VPINNs, we define the quadratic loss functional in terms of a Petrov-Galerkin-type variational formulation of the PDE problem: the trial space is a (Deep) Neural Network (DNN) manifold, while the test space is a finite-dimensional vector space. Whereas the VPINN's loss depends upon the selected basis functions of a given test space, herein, we minimize a loss based on the discrete dual norm of the residual. The main advantage of such a loss definition is that it provides a reliable and efficient estimator of the true error in the energy norm under the assumption of the existence of a local Fortin operator. We test the performance and robustness of our algorithm in several advection-diffusion problems. These numerical results perfectly align with our theoretical findings, showing that our estimates are sharp.\\

\end{abstract}
\setcounter{tocdepth}{3}
\tableofcontents
\section{Introduction}
The remarkable success of Deep Learning (DL) algorithms across different scientific areas \cite{hinton2012deep,krizhevsky2017imagenet,gheisari2017survey} in the last decade has recently led to explore the potential of this discipline to address classical problems in physics and mathematics. These problems include approximating the solutions of Partial Differential Equations (PDEs) employing (deep) Neural Networks (NN), which can accurately approximate continuous functions. The exponential growth of interest in these techniques started with the Physics Informed Neural Networks (PINNs) (\cite{raissi2019physics}). This method incorporates the governing physical laws the PDE describes in the learning process. Then, the network is trained on a dataset that consists of a random selection of points in the physical domain and its boundary. PINNs have been successfully applied to solve a wide range of problems in scientific computing, including fluid mechanics \cite{cai2021physics,mao2020physics}, wave propagation \cite{aldirany2023multi,rasht2022physics}, or inverse problems \cite{chen2020physics,mishra2022estimates}, among many others. However, the loss function in PINNs is given by the strong form of the PDE residual, and it is known that in some problems with low regularity of the data, the solution only makes sense in a variational form. Therefore, PINNs fails to provide accurate solutions in those cases \cite{krishnapriyan2021characterizing}. 

A natural continuation to PINNs addressing the aforementioned limitation is the so-called Variational PINNs (VPINNs) \cite{kharazmi2019variational}. Here, the authors introduce a variational formulation of the underlying PDE within a Petrov-Galerkin framework. The solution is then approximated by a (deep) NN, whereas the test functions belong to linear vector spaces. Finally, the authors define a variational loss function to minimize during the training process. However, this methodology was not as popular as its predecessor. The main reason might be that the approach is sensible to the choice of the basis test functions. That is, given a discrete test space, the method's loss, stability, and robustness heavily depends upon the choice of the basis functions of the given test space. In \cite{Berrone2022Solving}, authors present an a posteriori error analysis for discretizing elliptic boundary-value problems with VPINNs employing piecewise polynomials for the test space. As the loss function in VPINNs is, in general, not robust with respect to the true error (for example, the loss function can tend to zero even if the true error does not), they provide an error estimator employing classical techniques in finite element analysis to obtain practical information on the quality of the approximation. They also study in \cite{berrone2022variational} the importance of {employing} appropriate quadrature rules and the degree of piecewise polynomial test functions to guarantee optimal convergence rates for VPINNs when performing mesh refinements. Finally, in \cite{kharazmi2021hp}, the same authors from VPINNs introduced \textit{hp}-VPINNs. In the latter, they employ piecewise polynomials for testing and allow for $hp$-refinements via domain decomposition by selecting non-overlapping test functions. However, the aforementioned issue of {adopting} an appropriate loss function for each problem remains also in this approach. 

The strategy we propose in this article overcomes the aforementioned limitation by following the core ideas introduced in Minimum Residual (MinRes) methods. The latter is a class of numerical methods for solving PDEs that guarantee stability. The spirit of MinRes methods is to minimize the dual norm of the residual of the PDE given in variational form. Many strategies have been developed based on this idea over the last five decades, including the families of Galerkin Least-Squares methods \cite{bochev2009least,jiang1998least,hughes1989new}, First Order Least-squares \cite{cai1994first,cai1997first}, residual minimization methods on dual discontinuous Galerkin norms \cite{calo2020adaptive,cier2021automatically,cier2021nonlinear,rojas2021goal,kyburg2022incompressible}, isogeometric residual minimization methods \cite{calo2021isogeometric,los2020isogeometric,los2021isogeometric,los2021dgirm}, and Discontinuous Petrov-Galerkin (DPG) methods \cite{DEMKOWICZ20101558,demkowicz2013robust,demkowicz2014overview,munoz2021dpg,roberts2014dpg}. As the residual operator lives in the dual of the test space and the dual norm is difficult to compute, the use of the Riesz representation theorem \cite{oden2017applied} is natural in this context. The latter maps elements of the dual space into elements on the test space. Therefore, in many of the aforementioned methods, instead of minimizing the residual in the dual norm, they minimize its Riesz representative (which is an element of the test space) in a given test norm. 

In this work, we revisit the initial work on VPINNs from Kharazmi et al. in \cite{kharazmi2019variational}, and we introduce a Robust version of Variational Physics-Informed Neural Networks (RVPINNs). As in VPINNs, we consider a Petrov-Galerkin formulation of the PDE, approximate the solution employing NNs, and {introduce} a finite-dimensional test space. Our goal is to minimize the residual in the discrete dual norm. Therefore, we {introduce} a single test function that is the Riesz representation of the weak residual over the discrete test space. We {define} an inner product and span the Riesz representative on a discrete basis so the resulting loss function will include the inverse of the Gram matrix corresponding to the selected inner product. We prove that the norm of such test function is efficient and reliable, i.e., the norm of the true error is bounded from below and above (up to an oscillation term) by the norm of the residual representative. Moreover, as we show in the numerical results, our strategy is insensitive (unlike VPINNs) to the choice of the discrete basis spanning the test space (assuming that the numerical integration and the inversion of the Gram matrix are sufficiently accurate). It is easy to see that minimizing the proposed loss functional is equivalent to minimizing the test norm of the Riesz representation of the residual, which is the idea behind classical MinRes methods. Summarizing, we provide a general mathematical framework to define robust loss functionals in VPINNs. Our strategy relies on two ideas: (a) the appropriate selection of the inner product in the test space, ensuring the stability of the variational formulation, and (b) the {adoption} of a single test function, that is the Riesz representation of the weak residual. 

In particular, if we {consider} an orthonormal discrete basis with respect to the inner product in the test space, the Gram matrix becomes the identity, and we recover the original definition of the loss functional in VPINNs. The difference with classical VPINNs is that our strategy is robust and independent of the choice of the basis functions. Other works employ similar ideas closely related to minimum residual methods \cite{cai2022least,brevis2022neural}. {In \cite{ainsworth2021galerkin}, the authors approximate the solution of symmetric and positive definite problems employing the discrete weak residual of the variational problem. For that, they define an adaptive construction of a sequence of finite-dimensional subspaces whose basis functions are realizations of a sequence of neural networks.} In \cite{taylor2023deep}, the authors minimize the dual norm of the weak residual for symmetric variational formulations in $H^{1}$ and rectangular domains via a spectral decomposition. In \cite{taylor2023deep2}, the same authors extend this method to time-harmonic Maxwell equations in the context of symmetric $H(curl)$-formulations. Also, in \cite{uriarte2023deep}, in the context of Petrov-Galerkin formulations, the authors minimize the dual norm of the residual based on the concept of optimal testing form \cite{demkowicz2014overview}. They approximate both the solution of the variational problem and the corresponding optimal test functions employing NNs by solving two nested deep Ritz problems. More recently, in~\cite{badia2024finite}, the authors interpolate the neural networks onto finite element spaces to represent the (partial) unknowns. This strategy overcomes, in particular, the challenges related to the imposition of boundary conditions and approximations of the NN derivatives during training.

Finally, apart from the usual limitations of DL-based technologies (optimizer, integration, etc.), the largest bottleneck of RVPINNs at the moment is that in certain configurations, we need to invert the Gram matrix corresponding to the basis and inner product {of} the test space. We will explore how to optimize this step in the future. However, there are particular cases where the inversion of the Gram matrix is trivial (as in the spectral case) or easy to compute. For example, if we consider a strong variational formulation, the test inner product is $L^2$, and {employing} piecewise discontinuous polynomials for testing the Gram matrix becomes block diagonal. On the other hand, for parametric problems, the Gram matrix may remain the same for different values of the parameters. Therefore, the inversion can be done offline, being valid for a large class of problems. Parametric problems are essential to solve inverse problems \cite{gao2022physics}, so RVPINNs could be of great interest in this area. 

The article is organized as follows: In Section \ref{Sec:Pre}, we introduce the variational formulation of the model problem we consider in this article, the NN framework for approximation of the solution, and a brief overview of VPINNs. We also provide an alternative definition of VPINNs employing a single test function in the loss. Section \ref{Sec:RVPINNs} presents the methodology of RVPINNs and the connection with other methods based on VPINNs. Section \ref{sec:aposteriori} is devoted to the derivation of robust error estimates. We first prove that the residual representative is a reliable and efficient a posteriori error estimator of the true error in the sense of equivalent classes. Then, we demonstrate under the assumption of a local Fortin operator's existence that the true error is equivalent to the residual error estimator up to an oscillation term. {Finally, we derive an \textit{a priori} error estimate in the sense of quasi-optimizers.} In Section \ref{Sec:NR} we test our method in several 1D and 2D advection-diffusion problems, showing the robustness of the approach. Finally, Section \ref{Sec:Conc} summarizes the conclusions and future research lines.

\section{Preliminaries}\label{Sec:Pre} 

\subsection{Abstract framework}\label{sec:abstract}
Let $U$ and $V$ denote Hilbert spaces with norms $\|\cdot\|_U$ and $\|\cdot\|_V$. Assume that we are interested in obtaining an approximation of a PDE problem admitting a variational formulation of the form: 
\begin{equation}\label{eq:weakPDE}
\begin{array}{l}
\text{Find } u \in U,  \text{ such that: } \,  r(u \, , v) := l(v) - a(u\, , v) = 0,  \,  \forall \, v \in V,
\end{array}
\end{equation}
where $l(\cdot) \in V^\prime$ is a bounded linear form, with $V^\prime$ denoting the dual space of $V$; $a(\cdot\, ,\cdot)$ is a bounded inf-sup stable bilinear form in $U \times V$. That is, there exist constants $\mu, \alpha >0$, respectively, such that:
\begin{equation}\label{eq:bound_cont}
a(w\,,v) \leq \mu \|w\|_U\|v\|_V, \quad \forall \, w \in U, v \in V,
\end{equation}
and
\begin{equation}\label{eq:inf_sup_cont}
\sup_{0\neq v \in V}\dfrac{a(w\,,v)}{\|v\|_V} \geq \alpha \|w\|_U,  \quad  \forall \, w \in U.
\end{equation}
We also assume that, for all $v \in V$, the operator $a(\cdot\, , v) \in U^\prime$ satisfies:
\begin{equation}\label{eq:A_prime_kernel}
a(w\, , v) = 0, \, \forall \,  w\in U \Longrightarrow v = 0.     
\end{equation}
Problem~\eqref{eq:weakPDE} admits a unique solution by means of the Banach-Ne\v{c}as-Bab\v{u}ska Theorem, and the following a priori estimate holds (see~\cite[Theorem 1.1]{di2011mathematical}):
\begin{equation}\label{eq:a_piori_continuous}
\|u\|_U\leq \dfrac{1}{\alpha}\|l(\cdot)\|_{V^\prime},
\end{equation}
with 
\begin{equation}\label{eq:dual_norm_continuous}
\|l(\cdot)\|_{V^\prime} := \sup_{0\neq v \in V} \dfrac{l(v)}{\|v\|_V}.
\end{equation}
Additionally, it is well-known that~\eqref{eq:inf_sup_cont} and \eqref{eq:A_prime_kernel} imply that the following adjoint inf-sup condition holds (see, e.g., \cite[Theorem 1]{demkowicz2006babuvska}:
\begin{equation}\label{eq:inf_sup_cont_dual}
\sup_{0\neq w \in U}\dfrac{a(w\,,v)}{\|w\|_U} \geq \alpha \|v\|_V,  \quad  \forall \, v \in V.
\end{equation}
\begin{remark}[The weak residual]\label{rem:residual}
For a given $w \in U$, we refer to $r(w \,,\cdot) \in V^\prime$ (see~\eqref{eq:weakPDE}) as the weak residual.
\end{remark}
\subsection{Neural Network framework}
To numerically approximate \eqref{eq:weakPDE}, we consider a (Deep) Neural Network (DNN) function with input $\Vx = (x_1,\dots,x_d)$ and output $u_{\theta}(\Vx)$, where $\theta \in \mathbb{R}^S$ represents the trainable parameters. For simplicity, in this work, we employ a simple fully-connected feedforward Neural Network structure composed of $L$ layers. Each layer $l$ in the Neural Network consists of a set of neurons, which compute a weighted sum of their inputs plus a bias followed by a nonlinear activation function in the first $L-1$ layers and the identity function as the activation in the last layer. More precisely, the output of layer $l$, with $l=1,\dots L-1$, is given by:
\begin{equation}
\Vz^{(l)} = \sigma(\Vw^{(l)} \Vz^{(l-1)} + \Vb^{(l)}),
\end{equation}
where $\sigma$ is a nonlinear activation function (e.g., the tanh activation function), $\Vw^{(l)}$, $\Vb^{(l)}$ are the weights and biases respectively associated with the layer $l$, and $\Vz^{(0)} = \Vx$ is the input to the first layer. The final layer $L$ is given by:
\begin{equation}
u_{\theta} = \Vw^{(L)} \Vz^{(L-1)} +\Vb^{(L)}. 
\end{equation}
Using an optimization algorithm, the Neural Network weights and biases are learned from a training set by minimizing a loss functional of input-output pairs.\\
\noindent
We denote by $U_{NN}$ the manifold consisting of all possible realizations for a given DNN architecture belonging to $U$, where $U$ is defined in Section~\ref{sec:abstract}. That is, 
\begin{eqnarray}\label{eq:UNN}
U_{NN} :=\{ u_{\theta}, \, \forall \, \theta \in \mathbb{R}^S \}\cap U. 
\end{eqnarray}
{In the following, we assume $U_{NN}\neq \emptyset$}.
Finally, for a given discrete space $V_M \subseteq V$, we consider the following Petrov-Galerkin{-type} discretization of~\eqref{eq:weakPDE}:
\begin{equation}\label{eq:PG}
\begin{array}{l}
\text{Find } u_{NN} \in U_{NN},  \text{ such that: } \,  r(u_{NN},v_M)  = 0,  \,  \forall \, v_M \in V_M.
\end{array}
\end{equation}
{\begin{remark}[Existence of solutions for the Petrov-Galerkin-type discretization]\label{rem:existence_PG}
The non-emptiness for $U_{NN}$ is a necessary assumption to guarantee the existence of solutions to problem~\eqref{eq:PG}, as an inadequate DNN structure setup could generate a manifold $U_{NN}$ that is not contained in $U$. This is the case of a strong variational formulation set in $U\subset H^2(\Omega)$ (cf. Remark~\ref{rem:alternative_VF}) with activation functions that are not sufficiently smooth (e.g., \text{ReLU}). However, we also notice that non-emptiness is not a sufficient condition to guarantee the existence of solutions, as they could be meaningful only in a dense sense (i.e., when $S \rightarrow +\infty$, with $S$ denoting the number of trainable parameters). For instance, when considering smooth activation functions, but solutions are not smooth (cf. Section~\ref{sec:delta}). The last can be circumvented by introducing the concept of \textit{quasi-minimizers}, as we will see in Section~\ref{sec:quasi_mini}. 
\end{remark}
}
\subsection{Variational Physics-Informed Neural Networks}
Variational Physics Informed Neural Networks (VPINNs) approximate the solution $u$ of problem~\eqref{eq:weakPDE} by minimizing a loss functional defined in terms of the weak residual (see Remark~\ref{rem:residual}), and a discrete space $V_M\subseteq V$ with basis $\{\varphi_m\}_{m=1}^M$. That is, it approximates $u$ by solving the following minimization problem\footnote{{Note that this problem may not admit solutions as a consequence of Remark~\ref{rem:existence_PG}}}:
\begin{equation}\label{eq:loss_VPINNs_classical}
\text{Find } u_{\theta^\ast},  \text{ such that } \theta^\ast = \argmin_{\theta \in \mathbb{R}^S} \,  \mathcal{L}_r\left(u_{\theta} \right),
\end{equation}
with $\mathcal{L}_r\left(u_{\theta} \right)$ being a loss functional defined in terms of the weak residual and the discrete basis. This approach allows to consider scenarios where the solution belongs to less regular spaces than in collocation PINNs. For example, elliptic problems with solutions in $H^{2-\epsilon}(\Omega)$, with {$\Omega \subset \mathbb{R}^2$ and} $0<\epsilon \leq 1$.\\ 
A typical definition for the loss functional in ~\eqref{eq:loss_VPINNs_classical} found in the literature is the following (see, e.g., \cite{kharazmi2019variational, KHARAZMI2021113547,Berrone2022Solving, berrone2022variational}):
\begin{equation}\label{eq:loss_classical}
\mathcal{L}_{r}(u_{\theta}) := \sum_{m=1}^M r(u_{\theta}\, , \varphi_m)^2  + C(u_{\theta}),
\end{equation}
where $C(\cdot)$ is a $V_M$-independent {non-negative-valued cost} functional employed to impose {data constraints, including boundary conditions,} for $u_{NN}$ that can be taken as $C(u_{\theta})=0$ if boundary conditions are strongly imposed in the {DNN architecture, and no data for the interpolation is available.}
\begin{remark}[Classical approach]
We will refer to problem~\eqref{eq:loss_VPINNs_classical}, with the loss functional defined by~\eqref{eq:loss_classical}, as the classical VPINNs approach.
\end{remark}
\noindent
If we assume sufficient approximability of the {DNN architecture} and an ideal optimizer (solver), a minimizer for \eqref{eq:loss_classical} for which the loss functional vanishes is also a solution to the Petrov-Galerkin{-type} problem~\eqref{eq:PG}. However, adopting such a loss functional may be computationally inefficient. Indeed, given the test space $V_M$, Eq.~\eqref{eq:loss_classical} depends heavily upon the choice of its basis. In particular, a simple re-scaling of one basis function (e.g., by making it sufficiently large) may easily lead to catastrophic results (cf.~\cite{taylor2023deep}).
Consequently, there is a natural demand for robust loss definitions in practical VPINNs.
{\begin{remark}[Constraint assumption]
We assume $C(\cdot)$ in \eqref{eq:loss_classical} as a given cost functional since the present work focuses on the contribution in terms of the first term in~\eqref{eq:loss_classical}. The efficient implementation considering both contributions is a topic that requires further exploration.        
\end{remark}
}
{\subsection{An alternative definition of VPINNs}}
\noindent
To motivate how robust losses can be constructed, we first notice that, for a given set of trainable parameters $\theta$, we can define the following function in $V_M$:
\begin{equation}\label{eq:linear_comb_classical}
\widetilde{\varphi} = \sum_{m=1}^M r(u_\theta\, , \varphi_m)\, \varphi_m.
\end{equation}
Then, as a consequence of the linearity of the weak residual, it holds (cf.~\cite{kharazmi2019variational}):
\begin{equation}\label{eq:loss_VPINNs_general}
\mathcal{L}_r\left(u_{\theta} \right) = \sum_{m=1}^M r(u_{\theta}\, , \varphi_m)^2  + C(u_{\theta}) = r(u_{\theta}\, , \widetilde{\varphi})  + C(u_{\theta}).
\end{equation}
The last equation suggests that VPINNs can be generalized in the following way. For a given $\widetilde{\varphi} \in V_M$, a VPINNs loss functional can be defined as:
\begin{equation}\label{eq:loss_VPINNs_general_2}
\mathcal{L}_r\left(u_{\theta} \right)= r(u_{\theta}\, , \widetilde{\varphi}) + C(u_{\theta}).
\end{equation}
The goal now is to define the loss functional as ~\eqref{eq:loss_VPINNs_general_2} in terms of a particular discrete test function $\widetilde{\varphi} \in V_M$, {depending upon $u_\theta$}, such that VPINNs becomes robust. 
In the following Section, we present a strategy based on defining such a discrete function as the Riesz residual representative in $V_M$ with respect to the norm inducing the inf-sup stability~\eqref{eq:inf_sup_cont}.\\
\section{Robust Variational Physics-Informed Neural Networks}\label{Sec:RVPINNs}
In Robust Variational Physics-Informed Neural Networks (RVPINNs), we first introduce an inner product $(\cdot,\cdot)_{V}$\footnote{i.e.,  $||v||^2_{V}:= (v,v)_{V}$, for all $v \in V$.}, whose associated norm satisfies the inf-sup stability condition \eqref{eq:inf_sup_cont}. Next, for a given trainable parameter $\theta$, we compute $\phi:=\phi(\theta) \in V_M$ being the solution of the following Galerkin problem:
\begin{equation}\label{eq:Riesz}
(\phi,  \varphi_n)_{V} = r(u_{\theta}, \varphi_n), \text{ for } n=1,\dots,M,
\end{equation}
and define the loss functional as:
\begin{equation}\label{eq:loss_RVPINNs}
\mathcal{L}_r^{\phi}\left(u_{\theta} \right) := r(u_{\theta}, \phi) + C(u_{\theta}),
\end{equation}
where $C(\cdot)$ is a $V_M$-independent functional to impose {constraints} as in~\eqref{eq:loss_classical}. Finally, we obtain the approximation of the solution of problem~\eqref{eq:weakPDE} by solving the following minimization problem (cf., ~\cite{kharazmi2019variational}):
\begin{equation}\label{eq:RVPINNs}
\text{Find } u_{\theta^\ast},  \text{ such that } \theta^\ast = \argmin_{\theta \in \mathbb{R}^S} \,  \mathcal{L}_r^{\phi}\left(u_{\theta} \right),
\end{equation}
{when the DNN architecture allows for solutions for the associated Petrov-Galerkin-type problem (see Remark~\ref{rem:existence_PG}).}
\begin{remark}[Riesz representative of the weak residual]
The solution $\phi$ of problem~\eqref{eq:Riesz} is the Riesz representative in $V_M$ of the residual.
\end{remark}
\noindent
As a consequence of the linearity of the weak residual and the linearity of the inner product with respect to the second variable, defining: 
\begin{equation}\label{eq:linear_comb}
\phi := \sum_{m=1}^M \eta_m(\theta) \varphi_m, 
\end{equation}
problem~\eqref{eq:Riesz} leads to the resolution of the following problem written in matrix form:
\begin{equation}\label{eq:Riesz_mat}
G \eta(\theta) = {\cal R}(\theta),
\end{equation}
with $\eta(\theta)$ the vector of coefficients $\eta_m(\theta)$, $G$ the ($\theta$-independent) symmetric and positive definite Gram matrix of coefficients $G_{nm} = (\varphi_m,\varphi_n)_{V}$, and ${\cal R}(\theta)$ the vector of coefficients ${\cal R}_n(\theta) = r(u_{\theta}, \varphi_n)$, with $n=1,\dots,M$. Thus,
\begin{equation}\label{eq:residual_computation}
r(u_{\theta}\, , \phi) = \sum_{m=1}^M \eta_m(\theta) r(u_{\theta}\, , \varphi_m) = \sum_{m=1}^M {\cal R}_m(\theta) \eta_m(\theta) = {\cal R}(\theta)^T G^{-1} {\cal R}(\theta),
\end{equation}
implying that the loss functional \eqref{eq:loss_RVPINNs} is equivalently written as:
\begin{equation}\label{eq:AVPINNs_loss_mat}
\mathcal{L}^{\mathbf{\phi}}_{r}(u_{\theta}) = {\cal R}(\theta)^T G^{-1} {\cal R}(\theta) + C(u_{\theta}).
\end{equation}
We also notice that, by definition, to minimize the loss functional \eqref{eq:AVPINNs_loss_mat} is equivalent, up to the constraint $C(u_{\theta})$, to minimize the quantity $\|\phi\|_{V}^2$. Indeed, from \eqref{eq:Riesz}, we deduce:
\begin{eqnarray}\label{eq:norm_phi}
    \|\phi\|_{V}^2 = (\phi,\phi)_{V} = r(u_{\theta},\phi).
\end{eqnarray}
\begin{remark}[Gram matrix inversion]
   {A practical implementation of RVPINNs requires the computation of $r(u_\theta, \phi)$ following \eqref{eq:residual_computation}. Computing explicitly $G^{-1}$ could be computationally expensive in some scenarios, for instance, when considering standard FEM discrete test spaces defined in terms of piece-wise polynomials. Efficient inversions of the system \eqref{eq:Riesz_mat} may require iterative solvers and preconditioners (see, e.g., \cite{saad2003iterative}).
   }
\end{remark}
In Section~\ref{sec:fortin}, we prove that $\|\phi\|_{V}$ is, up to {the constraint term $C(u_\theta)$ and} an oscillation term, a local robust error estimation for the error $\|u-u_{\theta}\|_U$ under the assumption of the existence of a local Fortin's operator.
%
\subsection{Orthonormal discrete basis and relation with other VPINNs}
When a test space $V_M$ is the span of an orthonormal set $\{\varphi_m\}_{m=1}^M$ with respect to the $\|\cdot\|_{V}$-norm, the Gram matrix $G$ coincides with the identity matrix; therefore, the corresponding residual representative has the form:
\begin{equation}
\phi = \sum_{m=1}^M r(u_{\theta}, \varphi_m)\varphi_m,
\end{equation}
and the loss functional is explicitly written as 
\begin{equation}\label{eq:loss_RVPINNs_ortho}
\mathcal{L}^{\phi}_{r}(u_{\theta}) = \sum_{m=1}^M  r(u_{\theta}, \varphi_m)^2  + C(u_{\theta}).
\end{equation}
Thus, the classical VPINNs loss functional definition~\eqref{eq:loss_classical} is recovered for this case. This is also the case of the recently proposed Deep Fourier Residual method (see~\cite{taylor2023deep,taylor2023deep2}), where authors first define the loss functional as the continuous dual norm of the weak residual, that later is approximated by considering a truncation of a series in terms of an orthonormal basis for the test space. In particular, in~\cite{taylor2023deep}, authors consider diffusion problems and approximate the norm of the weak residual in $H^{-1}$. This turns out to be equivalent to~\eqref{eq:loss_RVPINNs_ortho} considering the standard $H^1$-norm for the space $V$, and the functions $\varphi_m$ as orthonormalized sinusoidal (spectral) functions with respect to the $H^1$-norm. 
\section{Error estimates for RVPINNs}\label{sec:aposteriori} 
One of the main complexities for proving the robustness of the residual estimator $\|\phi\|_{V}$ is that the solution of the Petrov-Galerkin problem~\eqref{eq:PG} may not have a solution or, if there exists, it may be non-unique since the space of all possible realizations of the NN {architecture} defines a manifold instead of a finite-dimensional space (see, e.g., \cite[Section 6.3]{berrone2022variational}). Thus, standard FEM arguments based on a discrete inf-sup condition cannot be applied in this context. 

Nevertheless, we can derive a posteriori error estimates using a different strategy. First, we introduce in Section~\ref{sec:aposteriori_classes} an equivalence class that allows us to neglect the part of the error that is $a$-orthogonal to $V_M$. For that equivalence class, we prove that the residual representative always defines a reliable and efficient a posteriori estimator for the error. Then, in Section~\ref{sec:fortin}, we consider the case of the full error, for which we demonstrate its equivalence to the residual error estimator up to an oscillation term and under the assumption of the existence of a local Fortin operator.
\subsection{A posteriori error estimates for RVPINNs in the sense of equivalence classes}\label{sec:aposteriori_classes}
We start by introducing the following Null space of the operator $A: U\mapsto V_M^\prime$, defined in terms of the bilinear form $a(\cdot\, , \cdot)$ associated with problem~\eqref{eq:weakPDE} (cf.~\cite{demkowicz2006babuvska}):
\begin{equation}\label{eq:U_kernel}
U_M^0 :=\left\{ w \in U \, : \,  \langle A(w) \, , v_M \rangle:= a(w\, ,\, v_M) = 0,\, \forall \, v_M \in V_M\right\},
\end{equation}
and the following norm for the quotient space $U /U_M^0$:
\begin{equation}\label{eq:quotient_norm_U}
\| [w]\|_{U/ U_M^0} := \inf_{w_0 \in U_M^0} \| w+w_0\|_U.
\end{equation}
We extend the definition of the bilinear form $a(\cdot\, , \cdot)$ to the product space  $U/ U_M^0 \times V$ as:
\begin{equation}\label{eq:bilinear_class}
    a([w]\, , v) = a(w\, , v), \quad \text{with } w \in [w] \text{ being any arbitrary representative of the equivalence class}. 
\end{equation}
The following result follows as a consequence of the boundedness and inf-sup stability properties of the bilinear form $a(\cdot\, , \cdot)$ (see equations~\eqref{eq:bound_cont} and \eqref{eq:inf_sup_cont}), respectively:
\begin{prop}\label{eq:infsup_bound_class}
The following boundedness and semi-discrete inf-sup conditions, respectively, are satisfied by the bilinear form~\eqref{eq:bilinear_class}:
\begin{equation}\label{eq:bound_class}
a([w],v_M) \leq \mu \| [w]\|_{U/ U_M^0} \|v_M\|_{V} , \quad \forall \, [w] \in U/ U_M^0, \, v_M \in V_M,
\end{equation}
\begin{equation}\label{eq:infsup_class}
\sup_{0\neq v_M \in V_M}\dfrac{a([w],v_M)}{\|v_M\|_{V}} \geq \alpha \| [w]\|_{U/ U_M^0}, \quad \forall \, [w] \in U/ U_M^0.
\end{equation}
\end{prop}
\begin{proof}
First, notice that ~\eqref{eq:bound_class} is direct from the definition \eqref{eq:quotient_norm_U} and the boundedness property of the bilinear form $a(\cdot\, , \cdot)$. Indeed,  for all $w \in U$ and $w_0 \in U_M^0$,  it holds:
\begin{equation*}
a(w\, , v_M) = a(w+w_0\,,  v_M) \leq \mu \|w+w_0\|_{U}\|v_M\|_{V}, \qquad \forall \, v_M \in V_M.
\end{equation*}
Thus,  \eqref{eq:bound_class} follows from taking the infimum.  To prove \eqref{eq:infsup_class},  first notice that, for all $v_M \in V_M$, \eqref{eq:inf_sup_cont_dual} implies:
\begin{equation*}
\alpha \|v_M\|_{V} \leq \sup_{0 \neq w \in U} \dfrac{a(w\,,v_M)}{\|w\|_U} = \sup_{[0] \neq [w] \in U / U_M^0} \dfrac{a([w]\,,v_M)}{\|w\|_U} 
\leq \sup_{[0] \neq [w] \in U / U_M^0} \dfrac{a([w]\,,v_M)}{\|[w]\|_{U / U_M^0}}, 
\end{equation*}
where the last inequality follows from the fact that $\|[w]\|_{U / U_M^0} \leq \|w\|_U$,  since $w_0 = 0 \in U_M^0$. Additionally, by construction, $\forall \,  [w]\in U/U_M^0$,
\begin{equation}
a([w]\, , v_M) = 0, \, \forall \,  v_M \in V_M \Longrightarrow [w] = [0]. 
\end{equation}
Therefore, \eqref{eq:infsup_class} follows as a consequence of Theorem 1 in \cite{demkowicz2006babuvska}.
\end{proof}
\noindent
Using the previous proposition, we can establish the main result of this section: 
\begin{thm}[Error class bounds in terms of the residual representative]\label{thrm:aposteriori}
Let $u \in U$ be the solution of the continuous problem~\eqref{eq:weakPDE}; $u_{\theta} \in U_{NN}$ denote a {DNN architecture} with trainable parameters $\theta \in \mathbb{R}^S$; $V_M \subseteq V$ denote a finite dimensional space with norm $\| \cdot \|_{V}$; and $\phi \in V_M$ be the solution of problem~\eqref{eq:Riesz}. It holds:
\begin{eqnarray}\label{eq:aposteriori}
\dfrac{1}{\mu} \|\phi\|_{V} \leq \|[u - u_{\theta}]\|_{U/U_M^0} \leq \dfrac{1}{\alpha} \|\phi\|_{V}.
\end{eqnarray}
\end{thm}
\begin{proof}
First notice that, by definition of $\phi$ and consistency of the analytical solution $u$, it holds:
\begin{eqnarray*}
\|\phi\|_{V} = \sup_{0\neq v_M \in V_M } \dfrac{(\phi, v_M)_{V}}{\|v_M\|_{V}} = \sup_{0\neq v_M \in V_M } \dfrac{r(u_{\theta}, v_M)}{\|v_M\|_{V}} = \sup_{0\neq v_M \in V_M } \dfrac{a([u-u_{\theta}] , v_M)}{\|v_M\|_{V}}.
\end{eqnarray*}
Therefore, the left inequality in~\eqref{eq:aposteriori} is obtained by using the boundedness property \eqref{eq:bound_class}. On another side, as a consequence of the inf-sup stability \eqref{eq:infsup_class}, it holds:
\begin{equation*}
\begin{array}{r@{\;}l}
\displaystyle \alpha \, \|[u-u_{\theta}]\|_{U/U_M^0}  \leq \displaystyle \sup_{0\neq v_M \in V_M} \dfrac{a([u-u_{\theta}],v_M)}{\|v_M\|_{V}} &= \displaystyle \sup_{0\neq v_M \in V_M} \dfrac{a(u-u_{\theta},v_M)}{\|v_M\|_{V}} \smallskip\\
& = \displaystyle \sup_{0\neq v_M \in V_M} \dfrac{r(u_{\theta},v_M)}{\|v_M\|_{V}} \smallskip\\
& = \displaystyle \sup_{0\neq v_M \in V_M} \dfrac{(\phi, v_M)_{V}}{\|v_M\|_{V}} \smallskip\\
& = \displaystyle \|\phi\|_{V},
\end{array}
\end{equation*}
proving the right inequality in \eqref{eq:aposteriori}. 
\end{proof}
\begin{remark}[Robustness of the residual representative in the sense of classes]\label{rem:robustness}    
Inequalities~\eqref{eq:aposteriori} show that, for any $\theta$, $\|\phi \|_{V}$ always, up to some multiplicative constants, defines an efficient (lower) and a reliable (upper) bound (thus, robust) estimation for the error $\|[u-u_{\theta}]\|_{U/U_M^0}$. This, in practice, implies that to minimize $\|\phi\|^2_{V}$ {$+ \, C(u_\theta)$} is equivalent to minimize $\|[u-u_{\theta}]\|^2_{U/U_M^0}$ {$+ \, C(u_\theta)$}.  Moreover, if we assume that problem~\eqref{eq:PG} admits a solution {and consider $C(u_\theta)=0$, $\|\phi\|_{V}\rightarrow 0^+$ implies that} $u-u_{\theta}$ converges to a function that belongs to $U_M^0$ as a consequence of the robustness. Thus, since $u \not\in U_M^0$, we conclude that $u_{\theta}$ converges to a function $u_{\theta^\ast}$ satisfying:
\begin{eqnarray*}
0=a(u-u_{\theta^\ast}\, ,v_M) = r(u_{\theta^\ast}\, ,v_M), \quad \forall \, v_M \in V_M.
\end{eqnarray*}
Therefore, $u_{\theta}$ converges to a  solution of the Petrov-Galerkin{-type} problem~\eqref{eq:PG}, as expected. 
\end{remark}
\subsection{Energy norm error estimates based on local semi-discrete inf-sup condition}\label{sec:fortin}
Even if the a posteriori error estimates of the previous section only ensure that the RVPINNs definition of the loss functional is robust in the sense of equivalence classes, numerical experiments confirm that it can also be robust concerning the true error in the energy norm.\\ 
We first notice that, as a consequence of Theorem~\ref{thrm:aposteriori} and the quotient norm definition~\eqref{eq:quotient_norm_U}, the following result can be deduced:
\begin{cor}[Lower bound for the error in the energy norm]\label{prop:aposteriori_lower}
Under the same hypothesis of Theorem~\ref{thrm:aposteriori}, it holds:
\begin{eqnarray}\label{eq:aposteriori_lower}
\dfrac{1}{\mu} \|\phi\|_{V} \leq \|u - u_{\theta}\|_{U}.
\end{eqnarray}
\end{cor}
\noindent
Therefore, for a given $\theta$, $\|\phi\|_{V}$ always defines an efficient bound for the error in the energy norm.\\ 
If we assume that the DNN {architecture} allows for solutions of the Petrov-Galerkin problem~\eqref{eq:PG}, we can also obtain a local reliable bound for $\|u-u_{\theta}\|_U$ through the following Assumption\footnote{{This Assumption is not standard; it is an adaptation of the classical Fortin's condition, commonly used in the discrete inf-sup stability analysis of mixed FEM formulations (see \cite{boffi2013mixed}) and, more recently, in the analysis of residual minimization-based FEMs, such as the DPG method (see \cite{nagaraj2017construction}).}}:
\begin{assumption}[Local Fortin's condition {I}]\label{as:fortin}
There exists a parameter $\theta^\ast \in \mathbb{R}^S$, such that $u_{\theta^\ast} \in U_{NN}$ is a solution of the Petrov-Galerkin problem~\eqref{eq:PG}. Additionally, there exists $R>0$ such that, for all $\theta \in B(\theta^\ast,R)$, there is an operator $\Pi_\theta: V\mapsto V_M$, and a $\theta$-independent constant $C_{\Pi}>0$, verifying:
\begin{enumerate}
\item[a)] $a(u_{\theta}\, , v -\Pi_\theta v) = 0$, \quad $\forall \, v \in V$,
\item[b)] $\|\Pi_\theta v\|_{V} \leq C_{\Pi} \|v\|_V$, \quad $\forall \, v \in V$,
\end{enumerate}
where $B(\theta^\ast,R)$ denotes an open ball of center $\theta^\ast$ and radius $R$, with respect to a given norm of $\mathbb{R}^S$.
\end{assumption}
\begin{prop}[Upper bound of the error in the energy norm]\label{prop:aposteriori_reliability_2}
Under the same hypothesis of Theorem~\ref{thrm:aposteriori}. If Assumption~\ref{as:fortin} is satisfied, it holds:
\begin{eqnarray}\label{eq:aposteriori_reliability_2}
\|u - u_{\theta}\|_{U}  \leq \dfrac{1}{\alpha} \osc(u) + \dfrac{1}{C_{\Pi}\alpha} \|\phi\|_{V}, \quad \forall \, \theta \in B(\theta^\ast,R),
\end{eqnarray}
with
\begin{equation*}
\osc(u) := \sup_{0 \neq v\in V}\dfrac{a(u,v-\Pi_\theta v)}{\|v\|_V}.
\end{equation*}
\end{prop}
\begin{proof}
If Assumption~\ref{as:fortin} is satisfied, it holds:
\begin{equation*}
\begin{array}{r@{\;}l}
\displaystyle \alpha \|u-u_{\theta}\|_{U} \leq \displaystyle \sup_{0\neq v \in V} \dfrac{a(u-u_{\theta},v)}{\|v\|_{V}} &= \displaystyle \sup_{0\neq v \in V} \dfrac{a(u-u_{\theta},v-\Pi_\theta v)}{\|v\|_{V}} + \sup_{0\neq v \in V} \dfrac{a(u-u_{\theta},\Pi_\theta v)}{\|v\|_{V}}\smallskip\\
&\leq \displaystyle \osc(u) + \dfrac{1}{C_{\Pi}}\sup_{0\neq v \in V} \dfrac{a(u-u_{\theta},\Pi_\theta v)}{\|\Pi_\theta v\|_{V}}\smallskip\\
&\leq \displaystyle \osc(u) + \dfrac{1}{C_{\Pi}}\sup_{0\neq v_M \in V_M} \dfrac{a(u-u_{\theta},v_M)}{\|v_M\|_{V}}\smallskip\\
& = \displaystyle \osc(u) + \dfrac{1}{C_{\Pi}}\sup_{0\neq v_M \in V_M} \dfrac{r(u_{\theta},v_M)}{\|v_M\|_{V}} \smallskip\\
& = \displaystyle \osc(u) + \dfrac{1}{C_{\Pi}}\sup_{0\neq v_M \in V_M} \dfrac{(\phi, v_M)_{V}}{\|v_M\|_{V}} \smallskip\\
& = \osc(u) + \dfrac{1}{C_{\Pi}}\displaystyle \|\phi\|_{V},
\end{array}
\end{equation*}
proving \eqref{eq:aposteriori_reliability_2}. 
\end{proof}

\begin{remark}[Existence of a local Fortin's operator]
First, notice that, for a given $\theta$, a Fortin operator's existence is only possible if $u_{\theta}\neq 0$ does not belong to the space $U_M^0$. Indeed, if $u_{\theta} \in U_M^0$, we have
\begin{equation*}
a(u_{\theta}\, , \Pi_\theta v) = 0, \text{  for all Fortin operator } \Pi_\theta,
\end{equation*}
since $\Pi_\theta v \in V_M$. Thus, in such a case, the residual representative $\|\phi\|_{V}$ is only a reliable estimate in the spirit of Theorem~\ref{thrm:aposteriori}. Nevertheless, it is expected that, in the minimization procedure, the NN solution approaches a solution of the Petrov-Galerkin problem~\eqref{eq:PG} (see Remark~\ref{rem:robustness}), implying that the NN solution will not belong to the space $U_M^0$ in a further iteration. Therefore, it is meaningful to assume that, after a sufficient number of iterations, the nonlinear solver will reach a parameter $\theta$ belonging to a neighborhood of a parameter $\theta^\ast$, in which $\theta^\ast$ is the local minimum, implying that a local semi-discrete inf-sup condition is satisfied, as stated in the following Lemma (cf.~\cite{fortin1977analysis}).
\end{remark}
\begin{lem}[Local Fortin's lemma]\label{lem:fortin}
If Assumption~\ref{as:fortin} is satisfied, then the following local semi-discrete inf-sup condition is satisfied:
\begin{equation}\label{eq:fortin_property}
\dfrac{\alpha}{C_\Pi} \|{u_\theta}\|_U \leq \sup_{0\neq v_M \in V_M} \dfrac{a({u_\theta}, v_M)}{\|v_M\|_{V}}, \quad \forall \, {u_\theta} \in U_{NN}^{\theta^\ast}:={\Span}\left\{ u_{\theta} : \theta \in B(\theta^\ast,R)\right\}.
\end{equation}
\end{lem}
\begin{proof} Using the properties of Fortin's operator and the linearity of the bilinear form $a(\cdot\, , \cdot)$ with respect to the first variable, it holds:
\begin{equation*}
 \sup_{0\neq v_M \in V_M} \dfrac{a({u_\theta}, v_M)}{\|v_M\|_{V}} \geq  \sup_{0\neq v \in V} \dfrac{a({u_\theta}, \Pi_\theta v)}{\|\Pi_\theta v\|_{V}} = \sup_{0\neq v \in V} \dfrac{a({u_\theta}, v)}{\|\Pi_\theta v\|_{V}} \geq \dfrac{1}{C_\Pi}\sup_{0\neq v \in V} \dfrac{a({u_\theta}, v)}{\|v\|_{V}} \geq \dfrac{\alpha}{C_\Pi}\|{u_\theta}\|_U.
 \end{equation*}
\end{proof}
\noindent
Finally, with the help of the previous Lemma, we can prove the following local a priori error estimate:
\begin{prop}[Local a priori error estimate I]\label{prop:local_apriori_1} If Assumption~\ref{as:fortin} is satisfied, it holds:
\begin{equation}\label{eq:local_inf_sup}
    \|u - u_{\theta^\ast}\|_U \leq \left(1 + \dfrac{\mu C_{\Pi}}{\alpha}\right) \inf_{{u_\theta}\in U_{NN}^{\theta^\ast}} \|u - {u_\theta}\|_U.
\end{equation}
\begin{proof}
First notice that, for all ${u_\theta} \in U_{NN}^{\theta^\ast}$, it holds
\begin{equation*}
   \dfrac{\alpha}{C_\Pi} \|u_{\theta^\ast}-{u_\theta}\|_U \leq \sup_{0\neq v_M \in V_M} \dfrac{a(u_{\theta^\ast}-{u_\theta},v_M)}{\|v_M\|_V} = \sup_{0\neq v_M \in V_M} \dfrac{a(u-{u_\theta},v_M)}{\|v_M\|_V}\leq \mu \|u-{u_\theta}\|_U.
\end{equation*}
Therefore, the result \eqref{eq:local_inf_sup} follows from the triangular inequality 
\begin{equation*}
    \|u-u_{\theta^\ast}\|_U \leq \|u-{u_\theta}\|_U + \|u_{\theta^\ast}-{u_\theta}\|_U.
\end{equation*}
\end{proof}
\end{prop}
\begin{remark}[Fortin's lemma converse] Following~\cite{ern2016converse}, it can be proved that the converse of Lemma~\ref{lem:fortin} is also valid. This could be useful in scenarios where proving the local semi-discrete inf-sup condition is simpler than proving the existence of a Fortin operator. {Note that the continuity constant of the Fortin operator is a measure of the loss of stability. Ideally, we would like to find a Fortin operator which is indeed an orthogonal projection, i.e., $C_{\Pi}=1$. However, the construction of a Fortin operator \cite{nagaraj2017construction} depends upon the formulation, the spaces, and the discretization. The construction of local Fortin operators for formulations involving NNs are out of the scope of this article.}
\end{remark}
\begin{remark}[Norm for the discrete space]
For the sake of simplicity, we assume that the norm for the discrete space $V_M$ coincides with the norm $\|\cdot\|_{V}$, of the space $V$. However, under certain conditions, it could be convenient to consider a different norm for the space $V_M$ being equivalent to $\|\cdot\|_{V}$. For instance, when properties of the discrete test allow obtaining a better inf-sup constant for the local semi-discrete inf-sup condition. In such a case, previous results are straightforwardly extended to this scenario.
\end{remark}
\subsection{{Error estimates in the quasi-minimizer sense}}\label{sec:quasi_mini}
{In the previous section, we derived a priori and a posteriori error estimates near the convergence regime by assuming the existence of a local solution of the Petrov-Galerkin-type problem~\eqref{eq:PG}. However, as mentioned in Remark~\eqref{rem:existence_PG}, this assumption cannot be satisfied in several practical scenarios. Additionally, we notice that our a priori error estimate does not take into consideration the contribution of the constraint $C(\cdot)$, since it is not necessarily true that a solution $u_{\theta^\ast}$ of the corresponding Petrov-Galerkin-type problem will also satisfy that $C(u_{\theta^\ast})$ will be bounded (up to a constant) by the infimum of $C(u_{\theta})$. This is the case when $C(\cdot)$ is defined in terms of noisy data and the DNN architecture allows for functions $u_\theta$ satisfying $C(u_\theta)=0$. Thus, a more general framework allowing for a less restrictive assumption must be considered. To this end, following~\cite{BREVIS2022}, we start by considering the following definition:
\begin{defn}[Quasi-minimizers] Let $\mathcal{L}: U \mapsto \mathbb{R}$ be a cost functional. Let $\delta_S>0$ and $U_{NN} \subset U$ be the manifold introduced in~\eqref{eq:UNN}, represented by a set $\theta \in \mathbb{R}^S$ of trainable parameters. A function $u_{\theta^S} \in U_{NN}$ is said to be a quasi-minimizer of $\mathcal{L}$ if the following holds true:
    \begin{equation}\label{eq:quasi_mini}
        \mathcal{L}(u_{\theta^S}) \leq \inf_{\theta \in \mathbb{R}^S} \mathcal{L}(u_{\theta}) + \delta_S.
    \end{equation}
\end{defn}
\noindent
The following is a direct consequence of the RVPINN cost function being positive-valued (cf. Eq.~\eqref{eq:norm_phi}).
\begin{prop}
For a given DNN structure $U_{NN}$, there exists a sufficiently small $\delta^\ast \geq 0$ and a function $u_{\theta^\ast} \in U_{NN}$ such that, for all $\delta_S > \delta^\ast$, $u_{\theta^\ast}$ is a quasi-minimizer of the cost functional~\eqref{eq:loss_RVPINNs}.
\end{prop}
}
\noindent
The following Lemma is a direct consequence of the quasi-minimizer definition and the property of the Riesz representative \eqref{eq:Riesz} (cf. proof of Theorem~\ref{thrm:aposteriori}):
\begin{lem}\label{lem:quasi}
    Let $u_{\theta^\ast}$ be a quasi-minimizer and $\delta^\ast >0$ be the smallest constant satisfying \eqref{eq:quasi_mini}. It holds
    \begin{equation}
        \left(\sup_{0\neq v_M \in V_M} \dfrac{a(u-u_{\theta^\ast},v_M)}{\|v_M\|_V}\right)^2 + C(u_{\theta^\ast})
        \leq \left(\sup_{0\neq v_M \in V_M} \dfrac{a(u-u_{\theta},v_M)}{\|v_M\|_V}\right)^2 +  C(u_{\theta}) + \delta^\ast, \quad \forall \, u_\theta \in U_{NN}.
    \end{equation}
\end{lem}
\noindent
We can now derive an a priori error estimate in the quasi-minimizer sense by considering the following relaxation of the Fortin's Assumption~\ref{as:fortin}\footnote{{This is indeed a relaxation of Assumption~\ref{as:fortin}, since in the case where there exists a solution of the Petrov-Galerkin-type problem~\eqref{eq:PG} and $C(u_\theta)=0$, we can consider $\delta_S=0$ and $\theta_S = \theta^\ast$.}}.
\begin{assumption}[Local Fortin's condition II]\label{as:fortin_2}
Given $\delta_S\ge0$ such that $u_{\theta_S} \in U_{NN}$ is a quasi-minimizer of~\eqref{eq:loss_RVPINNs}, there exists $R>0$ such that, for all $\theta \in B(\theta_S,R)$, an operator $\Pi_\theta: V\mapsto V_M$, and a $\theta$-independent constant $C_{\Pi}>0$, verifying the same properties of Assumption~\ref{as:fortin}.
\end{assumption}
\begin{prop}[Local a priori error estimate II] Let $u_{\theta^\ast}$ be a quasi-minimizer and $\delta^\ast >0$ be the smallest constant satisfying \eqref{eq:quasi_mini}. If Assumption~\ref{as:fortin_2} is satisfied, it holds:
\begin{equation}\label{eq:local_inf_sup_2}
    \|u - u_{\theta^\ast}\|^2_U + \left(\dfrac{\sqrt{2} C_{\Pi}}{\alpha}\right)^2C(u_{\theta^\ast}) \leq  \inf_{u_\theta\in U_{NN}^{\theta^\ast}} \left\{2\left(1 + \left(\dfrac{2\mu C_{\Pi}}{\alpha}\right)^2\right)\|u - u_\theta \|^2_U + \left(\dfrac{\sqrt{2}C_{\Pi}}{\alpha}\right)^2\big( C(u_{\theta}) + \delta^\ast \big)\right\} .
\end{equation}
\end{prop}
\begin{proof}
    The proof is similar as in Proposition~\ref{prop:local_apriori_1} by noticing that, for all $u_\theta \in U_{NN}^{\theta^\ast}$, with $U_{NN}^{\theta^\ast}$ defined as in \eqref{eq:fortin_property}, it holds
\begin{equation*}
\begin{array}{r@{\;}l}
   \left(\dfrac{\alpha}{2C_\Pi}\right)^2 \|u_{\theta^\ast}-u_\theta\|^2_U + \dfrac{1}{2}C(u_{\theta^\ast}) & \displaystyle \leq \frac{1}{4}\left(\sup_{0\neq v_M \in V_M} \dfrac{a(u_{\theta^\ast}-u_\theta,v_M)}{\|v_M\|_V}\right)^2  + \dfrac{1}{2}C(u_{\theta^\ast}) \smallskip \\
    & \displaystyle \leq \frac{1}{4}\left(\sup_{0\neq v_M \in V_M} \dfrac{a(u-u_\theta,v_M)}{\|v_M\|_V} + \sup_{0\neq v_M \in V_M} \dfrac{a(u-u_{\theta^\ast} ,v_M)}{\|v_M\|_V}\right)^2  + \dfrac{1}{2}C(u_{\theta^\ast}) \smallskip \\
    & \displaystyle \leq \left(\sup_{0\neq v_M \in V_M} \dfrac{a(u-u_\theta ,v_M)}{\|v_M\|_V}\right)^2  + \dfrac{1}{2}\left(C(u_\theta) + \delta^\ast\right),
   \end{array}
\end{equation*}
where the last inequality follows from the Young's inequality $2xy \leq x^2 + y^2, \, \forall x,y \in \mathbb{R}$, and Lemma~\ref{lem:quasi}.
\end{proof}
\section{Numerical examples}\label{Sec:NR}
To illustrate the proposed strategy's performance, we consider several diffusion-advection problems subject to Dirichlet-type boundary conditions. We focus on problems that either are challenging for the standard finite element method (advection-dominated problems) or do not admit a solution belonging to $H^2(\Omega)$. We consider two distinct discrete test spaces to construct the loss functionals: one employing standard FE piece-wise polynomials and another employing spectral orthonormal test functions. Here, we do not intend to compare the performance between different test spaces, but rather, we aim to numerically validate the robustness of RVPINNs.
\subsection{Diffusion-advection model problem}
For $d\geq1$, given a bounded and open Lipschitz polyhedra $\Omega \in \mathbb{R}^d$ with boundary $\partial \Omega$, this work considers the following linear diffusion-advection problem:\\ 
Find $u$, such that:
\begin{equation}\label{eq:strongDAR}
\begin{array}{r@{\;}l}
- \nabla \cdot \left(\varepsilon \nabla u\right) + \mathbf{\beta} \cdot \nabla u  &= f,  \quad \text{ in } \Omega,  \smallskip \\
u &= 0, \quad \text{ on } \partial \Omega,
\end{array}
\end{equation}
where $\varepsilon>0$ denotes a diffusive coefficient term,
$\mathbf{\beta} \in [L^\infty(\Omega)]^d$ is a divergence-free (i.e., $\text{div}\beta \equiv 0$) advective coefficient function. Finally, $f \in V^\prime$ is a given source term, with $V^\prime:= H^{-1}(\Omega)$ being the dual space of $V:=H_0^1(\Omega)$. Problem~\eqref{eq:strongDAR} admits the following continuous variational formulation:
\begin{equation}\label{eq:weakDAR}
\text{Find } u \in U:= V \, : \,r(u\, ,v) := l(v)-a(u\, ,v), \, \forall \, v \in V,
\end{equation}
with
\begin{equation}\label{eq:a_l_DAR}
\displaystyle a(u\, ,v) := {\left( \mathbf{\varepsilon} \nabla u \, , \nabla v\right)_0 + \left(\mathbf{\beta} \cdot \nabla u \, , v\right)_0} , \quad \text{and} \quad l(v) = \langle f\, ,  v\rangle,
\end{equation}
where $(\cdot\, , \cdot)_0$ denotes the $L^2$-inner product\footnote{We adopt the same notation for the scalar and vectorial $L^2$-inner products for the sake of simplicity.}, and $\langle\cdot \, , \cdot\rangle$ denotes the duality map between $V^\prime$ and $V$, coinciding with $(f\, , v)_0$ when $f$ belongs to $L^2(\Omega)$. \\
Let us equip the Hilbert spaces $U, V$ with the norms 
\begin{equation}\label{eq:DAR_norms}
\|\cdot\|^2_U:= \| \cdot\|^2_V:= \varepsilon \,(\nabla \cdot\, , \nabla \cdot)_0.
\end{equation}
Using the Cauchy-Schwarz inequality, the following boundedness estimate is proved (cf.~\eqref{eq:bound_cont}):
\begin{equation}\label{eq:boundedness_DAR}
a(w \, , v) \leq \left(1+\dfrac{C_\Omega\, \|\beta\|_{\infty}}{\varepsilon} \right) \, \|w\|_U \|v\|_V, \quad \forall \, w \in U, v \in V,
\end{equation}
where $C_\Omega$ denotes the Poincar\'e constant (see~\cite{payne_weinberger_1960,acosta2004optimal}). Moreover, as a consequence of the fact that (recall that $\text{div} \beta \equiv 0$):
\begin{equation}\label{eq:div_cond}
    (\beta\cdot\nabla v \, , v)_0 = 0, \quad \forall \, v \in V,
\end{equation}
a straightforward verification shows that the following coercive estimate holds: 
\begin{equation}\label{eq:coercivity_DAR}
a(v \, , v) \geq \|v\|^2_V, \quad \forall \, v\in V.
\end{equation}
Notice also that the coercive property~\eqref{eq:coercivity_DAR} implies that the following inf-sup condition is satisfied (cf.~\eqref{eq:inf_sup_cont}):
\begin{equation}\label{eq:inf_sup_DAR}
    \sup_{0\neq v \in V} \dfrac{a(w\,,\,v)}{\|v\|_V} \geq \|w\|_U, \quad \forall \,  w \in U.
\end{equation}
\begin{remark}[Alternative Variational Formulations]\label{rem:alternative_VF}
With the spirit of considering a favorable scenario for spectral and FEM test spaces, we consider a standard weak variational formulation with the same test and trial spaces. However, we recall that the strategy is general, allowing for taking into account alternative formulations as a starting point. As an illustrative example, consider the pure diffusive particular case of problem \eqref{eq:strongDAR} (i.e., with $\varepsilon = 1$, $\mathbf{\beta} = {\bf 0}$). Then, the following two alternative variational formulations of the form~\eqref{eq:weakDAR} are admissible:
\begin{enumerate}
    \item[a)] Strong VF: $U=H^2(\Omega) \cap H^1_0(\Omega)$,  $V=L^2(\Omega)$,  $a(u,v) = -\left( \Delta u \, , v\right)_0$ \, , $l(v) = \langle f\, ,  v\rangle$.
    \item[b)] Ultraweak VF: $U=L^2(\Omega)$, $V = H^2(\Omega) \cap H^1_0(\Omega)$,  $a(u,v) = -\left( u \, , \Delta v\right)_0$ \, , $l(v) = \langle f\, ,  v\rangle$.
\end{enumerate}
Moreover, it can be proved that conditions \eqref{eq:bound_cont}-\eqref{eq:inf_sup_cont} are satisfied by endowing the Hilbert space $H^2(\Omega) \cap H^1_0(\Omega)$ with the norm $\left( \Delta \cdot \, , \Delta \cdot \right)_0^{1/2}$, and $L^2(\Omega)$ with its standard norm.
\end{remark}
\begin{subsection}{{A weak BCs imposition for 1D problems}}\label{sec:weak}
Setting $\Omega = (-1,1)$, $U = H^1(\Omega)$, $V = H_0^1(\Omega)$, and $\widehat{u}=(u(-1), u(1))$, the variational formulation is defined as follows\footnote{{Notice that this way of imposing BCs can also be applied to other alternative variational formulations. For instance, the Strong VF in the spirit of Remark~\ref{rem:alternative_VF}.}}
\begin{equation}\label{eq:weak_BCs_VF}
    \text{Find } u \in U : a(u,(v,\widehat{v})):= ( \varepsilon u^\prime \, ,  v^\prime)_0 + (\beta u^\prime \, , v)_0 + \widehat{u} \cdot \widehat{v} = l((v,\widehat{v})) := \langle f\,, v \rangle, \, \forall \, (v,\widehat{v}) \in V \times \mathbb{R}^2. 
\end{equation}
Variational formulation ~\eqref{eq:weak_BCs_VF} is well-possed. Indeed, defining $\|\cdot\|^2_V:= \varepsilon \,( \cdot^\prime\, , \cdot^\prime)_0$ and denoting by $\|\cdot\|_E$ the Euclidean norm of $\mathbb{R}^2$, we can consider the norms
\begin{equation}
    \|u\|_{U}^2 := \|u\|_{V}^2 + \|\widehat{u}\|_{E}^2 \quad \text{and} \quad \|(v,\widehat{v})\|_{V\times\mathbb{R}^2}^2 := \|v\|_{V}^2 + \|\widehat{v}\|_{E}^2,
\end{equation}
for the trial and test space, respectively. A straightforward verification proves that the boundedness condition~\eqref{eq:bound_cont} is satisfied with the upper bound constant $\mu$ of Eq. \eqref{eq:boundedness_DAR}. 
It is also immediate to verify that the condition \eqref{eq:A_prime_kernel} is satisfied. Indeed, noticing that $(\beta v^\prime, v)_0 = 0$, for all $v\in V$ (cf. Eq.~\eqref{eq:div_cond}), if $a(u,(v,\widehat{v})) =0 $ holds for all $u \in U$, then it holds for $u = v + u_{\widehat{v}}$, with the last being the (unique) solution of the problem:
\begin{equation*}
    \text{Find } u_{\widehat{v}} \in U : -\varepsilon u_{\widehat{v}}^{\prime \prime} + \beta u_{\widehat{v}}^{\prime} = 0, \quad u_{\widehat{v}}(-1) = \widehat{v}(1), \quad u_{\widehat{v}}(1) = \widehat{v}(2),
\end{equation*}
implying that $ \|(v,\widehat{v})\|_{V\times\mathbb{R}^2}$ must be equal to zero. Finally, we can prove that the inf-sup condition ~\eqref{eq:inf_sup_cont} is satisfied with $\alpha = 1$ by first noticing that it is trivially satisfied for all $u \in V$. In addition, for a given $u$ such that $\widehat{u}\neq0$, given $v = u - u_d$ with $u_d$ any lifting of the boundary data $\widehat{u}$, and $\widehat{v} = \widehat{u} - \eta \dfrac{\widehat{u}}{\|\widehat{u}\|^2_E}$, with $\eta = -(\varepsilon u^\prime, u_d^\prime)_0 + (\beta u^\prime, u - u_d)_0$, it holds
\begin{equation*}
    \displaystyle \frac{a(u,(v,\widehat{v}))}{\|(v,\widehat{v})\|_{V\times\mathbb{R}^2}} = \|u\|_{U}.
\end{equation*}
The result follows from noticing that the supreme over $V\times\mathbb{R}^2$ is always larger or equal to the left-hand side of the above equality.\\
With this alternative definition, it can be proved that the corresponding loss functional is given by
\begin{equation}\label{eq:loss_weak_weak}
    \mathcal{L}^{\phi}_{r}(u_{\theta}) = \|\phi\|_V^2  + |u_\theta(-1)|^2 + |u_\theta(1)|^2 + C(u_{\theta}),
\end{equation}
with $\phi$ being the residual representative in $V_M \subseteq V$ with residual $r(u,\cdot)$ defined as in the variational formulation \eqref{eq:weakDAR}. Indeed, choosing the canonical basis for $\mathbb{R}^2$, the corresponding (block-diagonal) Gram matrix $\widehat{G}$ associated to the inner product inducing the norm for the space $V\times\mathbb{R}^2$, and the rhs $\widehat{\cal R}(\theta)$, respectively, have the following structure
\begin{equation}\label{eq:Gram_full_weak}
 \widehat{G}  = \left(\begin{matrix}
        G & \\
         & I_2
    \end{matrix}\right) \quad \text{and} \quad \widehat{\cal R}(\theta) = \left(\begin{matrix}
         {\cal R}(\theta)\\
         -u_\theta(-1) \\
         -u_\theta(1)
    \end{matrix}\right),
\end{equation}
with $I_2$ denoting the $2\times2$ identity matrix, $G$ the Gram matrix associated with the inner product of $V$, and ${\cal R}(\theta)$ the rhs associated with the residual $r(u_\theta,\cdot):=\langle f\,, \cdot \rangle -( \varepsilon u^\prime \, ,  \cdot^\prime)_0 - (\beta u^\prime \, , \cdot)_0$ in $V^\prime$.
\begin{remark}[Penalized weak BCs imposition]
We notice that it is also possible to consider a penalization for the weak imposition by multiplying the term $\widehat{u}\times\widehat{v}$ with a (possibly boundary-dependent) positive penalization term. This could serve as a further stabilization strategy in complex scenarios. For instance, in the advection-dominated regime.
\end{remark}
\end{subsection}
\subsection{Discrete setting}\label{sec:discrete_setting}
{For the sake of simplicity, in all the numerical examples we assume that there is no data available for interpolation; thus, we consider $C(u_\theta)=0$. For one-dimensional problems, we} set $\Omega:= (-1,1)$ and define two kinds of test spaces. First, we consider $V_M$ as FE space of the standard globally continuous and piece-wise linear functions that belong to $H_0^1(\Omega)$, defined over a uniform mesh partition of $\Omega$ into $M+1$ elements $[x_i, x_{i+1}]$, with $x_0=-1<x_1<\dots<x_M<x_{M+1}=1$. {Similarly, we consider a 2D case where $\Omega:= (0,1)^2$. Here, we employ tensor-product piece-wise linear FEM test functions}. The corresponding Gram matrix {in both cases} is constructed using the inner product, inducing the norm $\|\cdot\|_V$, defined in~\eqref{eq:DAR_norms}. {Herein, we compute $G^{-1}$ explicitly in both 1D and 2D examples employing FEM test spaces.} 

\noindent The second test space is intended to show examples involving an orthonormal basis for the discrete space $V_M$. Here, we normalize with respect to the inner product~\eqref{eq:DAR_norms} the first $M$ eigenfunctions of the Laplacian satisfying the Dirichlet homogeneous boundary conditions. The $m$-th eigenfunction and the corresponding orthonormal set, respectively, are given by
\begin{equation}\label{eq:spectral_1D}
s_m := \sin\left(m\pi \dfrac{x+1}{2}\right), \qquad E_M = \left\{\varphi_m := \dfrac{2 s_m}{\sqrt{\varepsilon} \pi m}  \right\}_{m=1}^M.
\end{equation}
Therefore, the loss functional~\eqref{eq:loss_RVPINNs_ortho} is explicitly given as:
\begin{equation}\label{eq:loss_DAR_spectral_1D}
\mathcal{L}^{\phi}_{r}(u_{\theta}) = \dfrac{4}{\varepsilon \pi^2}\sum_{m=1}^M  \dfrac{1}{m^2} \, r\left(u_{\theta}\,  ,  s_m \right)^2.
\end{equation}
\noindent
We also consider a two-dimensional numerical example involving an orthonormal basis for the discrete test space. In such a case, we set $\Omega:=(0,1)^2$ and construct the first $M$ and $N$ eigenfunctions in the $x$ and $y$ coordinates, respectively.  Here, the $m,n$-th eigenfunction and the corresponding orthonormal set, respectively, are given by
\begin{equation}\label{eq:spectral_2D}
s_{m,n} := \sin\left(m\pi x\right)\sin\left(n\pi y\right), \qquad E_{MN} = \left\{\varphi_{m,n} := \dfrac{s_{m,n}}{\sqrt{\varepsilon} \pi (m+n)}  \right\}_{m=1,n=1}^{M,N}.
\end{equation}
Thus,  the loss function~\eqref{eq:loss_RVPINNs_ortho} is explicitly given as:
\begin{equation}\label{eq:loss_DAR_spectral_2D}
\mathcal{L}^{\phi}_{r}(u_{\theta}) = \dfrac{1}{\varepsilon \pi^2}\sum_{n=1}^N \sum_{m=1}^M  \dfrac{1}{n^2+m^2} \, r\left(u_{\theta}\,  ,  s_{m,n} \right)^2.
\end{equation}
\noindent
For the architecture of the DNN approximator $u_\theta$, we adopt the following structure for the 1D problems: a fixed feed-forward fully connected network comprising five layers, each with 25 neurons. Throughout all layers, the activation function employed is $\tanh$. When strongly enforcing the homogeneous Dirichlet boundary conditions, we multiply the last layer's output by $(x+1)(x-1)$. Another version includes the {weak} imposition of homogeneous Dirichlet boundary conditions by adding to the functional loss the following quantity (see Section~\ref{sec:weak}):
\begin{equation}\label{eq:weak_bcs}
|u_\theta(-1)|^2 + |u_\theta(1)|^2.
\end{equation}
\noindent
As our choice for the nonlinear solver, we opt for the ADAM optimizer (refer to~\cite{kingma2014adam}) and initialize it with a learning rate of $0.0005$. Our nonlinear solver undergoes up to $6000$ iterations (epochs).
To approximate the integral terms appearing in the loss functional definitions of the 1D problems, we employ (a) a fixed Gaussian quadrature rule with five nodes per element when considering FE functions, and (b) a trapezoidal rule with $4000$ equally-distributed nodes when considering spectral functions. Lastly, the error $|u - u_\theta|_U$ is numerically estimated using a trapezoidal rule, with $10000$ equally-distributed nodes. For the 2D example, we consider a fixed feed-forward fully connected network comprising four layers, each with 40 neurons and the same type of activation function as in the 1D cases. For the training {in the spectral case}, $500 \times 500$ equally-distributed integration points, while equally-distributed $1000 \times 1000 $ integration points for estimating the error $|u - u_\theta|_U$. We strongly impose the Dirichlet boundary conditions by multiplying the last layer's output by $x(x-1)y(y-1)$.
\subsection{Smooth diffusion problems}
We first consider a simple diffusion problem with a smooth solution. More precisely, setting $\Omega = (-1,1)$, we consider the variational problem \eqref{eq:weakDAR} with $\varepsilon=1$, $\beta = 0$ and forcing term $f$ defined in such a way that the analytical solution is given by
\begin{eqnarray*}
u(x) = x\sin(\pi(x+1)).
\end{eqnarray*}
Figure~\ref{fig:smooth_spectral_strong_50} shows the numerical results obtained with RVPINNs with $50$ sinusoidal test functions and strong imposition of the BCs. We observe a good performance of the RVPINNs strategy, obtaining an accurate approximation of the analytical solution, as shown in Figure~\ref{fig:smooth_spectral_a}. Moreover, we observe a perfect match between the quantities $\|u-u_\theta\|_U$, $\|\phi\|_V$, and $\sqrt{\mathcal{L}^\phi_r(u_\theta)}$ as shown in figures~\ref{fig:smooth_spectral_b} and ~\ref{fig:smooth_spectral_c} where, in particular, the last one shows a perfect correlation between the true error and the square root of the loss functional. Here, we compute $\|\phi\|_V$ employing the right-hand-side of \eqref{eq:residual_computation}. Note that $\|\phi\|_V$ and $\sqrt{\mathcal{L}^\phi_r(u_\theta)}$ coincide when strongly imposing BCs. Figure~\ref{fig:smooth_FEM_strong_100} shows an almost identical behavior when considering $100$ FE test functions and strong imposition of BCs. We obtain similar results when considering the weak imposition of BCs, but we omit them here for brevity.\\
Note that in the case of pure diffusive problems, the continuity constant $\mu$ and the inf-sup stability constant $\alpha$ are equal to one when considering the $H^1$-seminorm (cf.~\eqref{eq:boundedness_DAR} and \eqref{eq:inf_sup_DAR}). Therefore, $\|\phi\|_V$ always defines a lower bound for the error (cf.~\eqref{eq:aposteriori_lower}) and an upper bound up to a constant close to one if the oscillation term in \eqref{eq:aposteriori_reliability_2} is sufficiently small. These theoretical considerations are confirmed in figures~\ref{fig:smooth_spectral_strong_50} and \ref{fig:smooth_FEM_strong_100}.\\
Finally, to show the effect of the oscillation term \eqref{eq:aposteriori_reliability_2}, Figure~\ref{fig:smooth_FEM_strong_5} shows the results obtained with only five test FE basis functions. Figure~\ref{fig:smooth_FEM_strong_5_b} reveals that the estimation is not robust after (approximately) $100$ iterations and only becomes a lower bound, as the DNN has reached (up to implementation precision) a solution of the Petrov-Galerkin problem ~\eqref{eq:PG}.
\begin{figure}
    \centering        
     \begin{subfigure}[b]{0.3\textwidth}
         \centering
         \includegraphics[width=\textwidth]{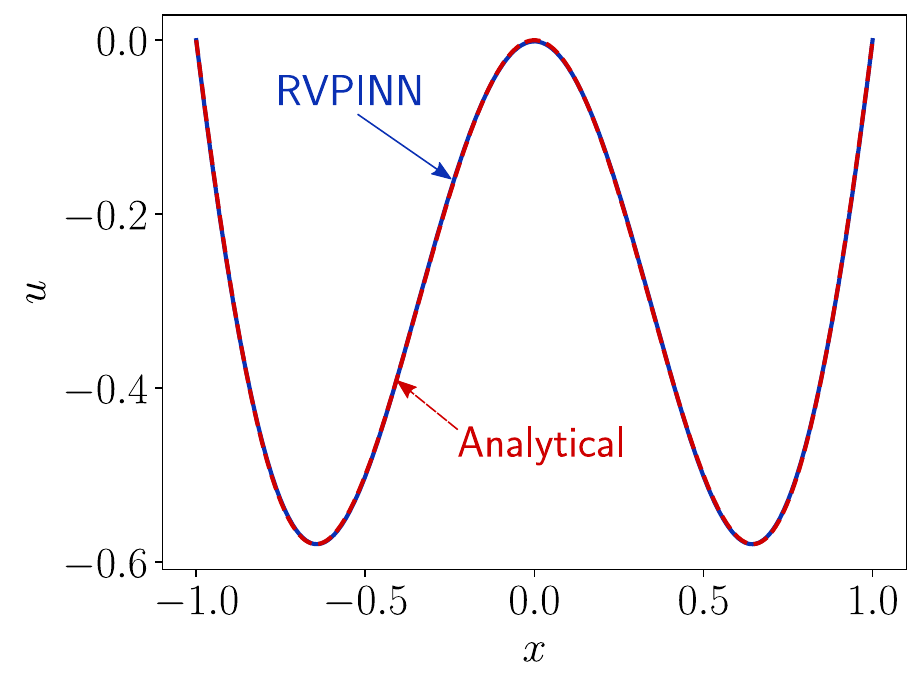}
         \caption{DNN best approximation}
         \label{fig:smooth_spectral_a}
     \end{subfigure}
     \hfill
     \begin{subfigure}[b]{0.3\textwidth}
         \centering
         \includegraphics[width=\textwidth]{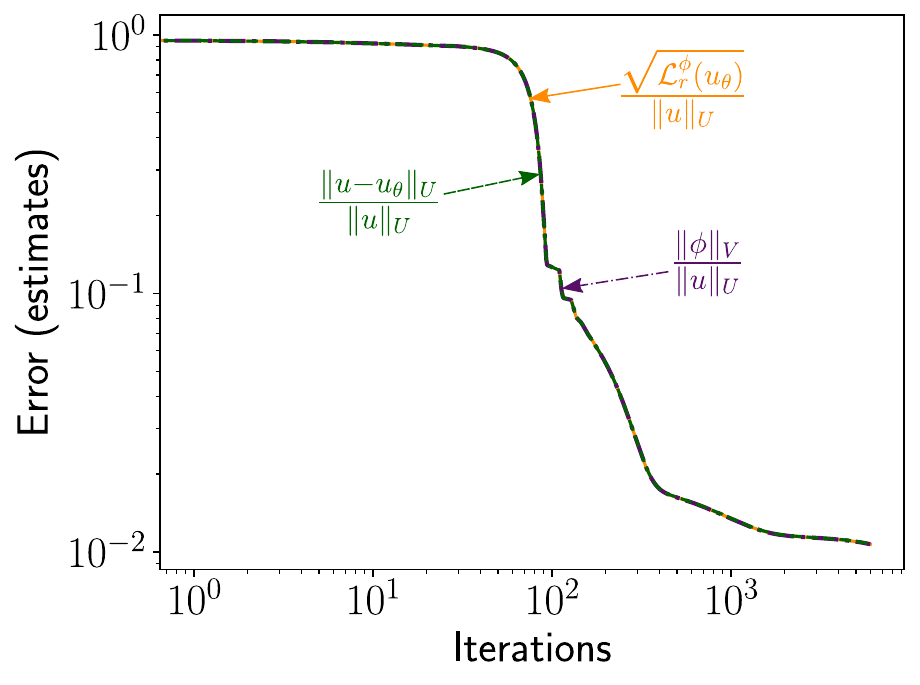}
          \caption{Estimates vs Iterations}
         \label{fig:smooth_spectral_b}
     \end{subfigure}
     \hfill
     \begin{subfigure}[b]{0.3\textwidth}
         \centering
         \includegraphics[width=\textwidth]{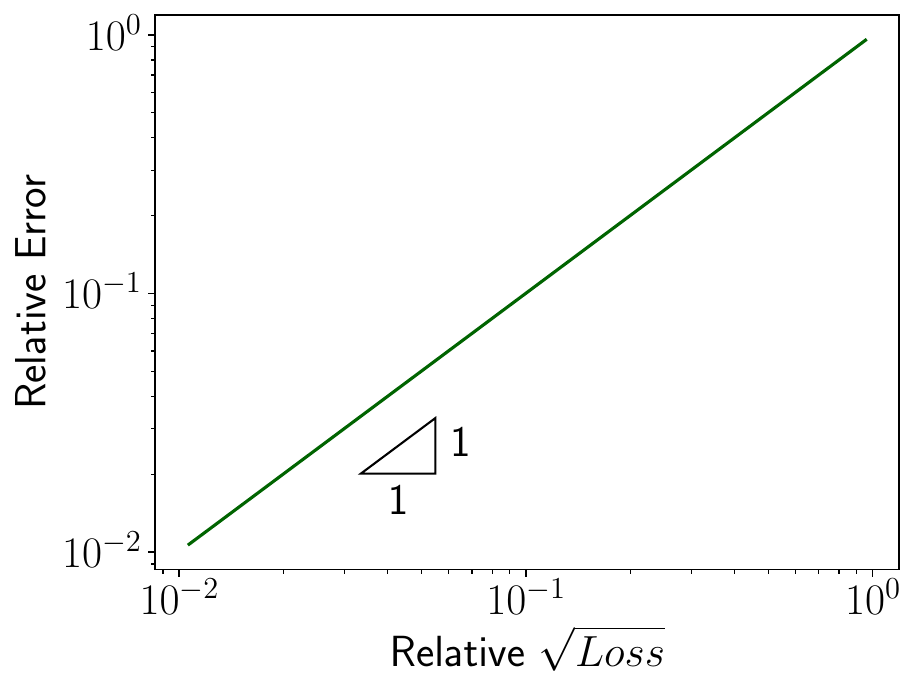}
         \caption{Error-Loss comparison}
         \label{fig:smooth_spectral_c}
     \end{subfigure}
     \hfill
    \caption{RVPINNs approximation of the smooth diffusion problem; strong BCs imposition, $50$ spectral test functions.}
    \label{fig:smooth_spectral_strong_50}
\end{figure}

\begin{figure}
    \centering        
     \begin{subfigure}[b]{0.3\textwidth}
         \centering
         \includegraphics[width=\textwidth]{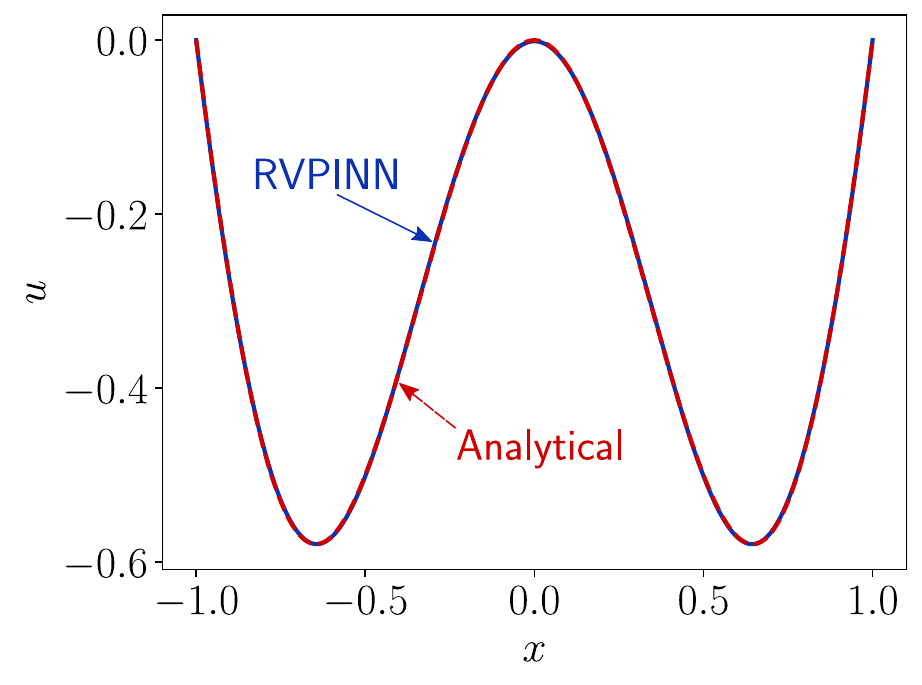}
         \caption{DNN best approximation}
         \label{fig:smooth_FEM_a}
     \end{subfigure}
     \hfill
     \begin{subfigure}[b]{0.3\textwidth}
         \centering
         \includegraphics[width=\textwidth]{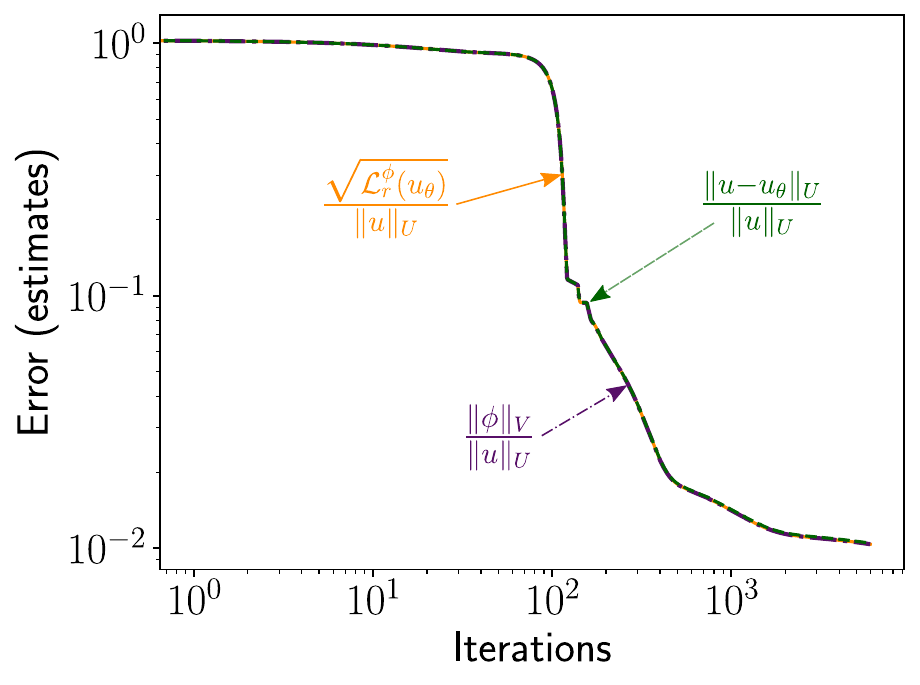}
         \caption{Estimates vs Iterations}
         \label{fig:smooth_FEM_b}
     \end{subfigure}
     \hfill
     \begin{subfigure}[b]{0.3\textwidth}
         \centering
         \includegraphics[width=\textwidth]{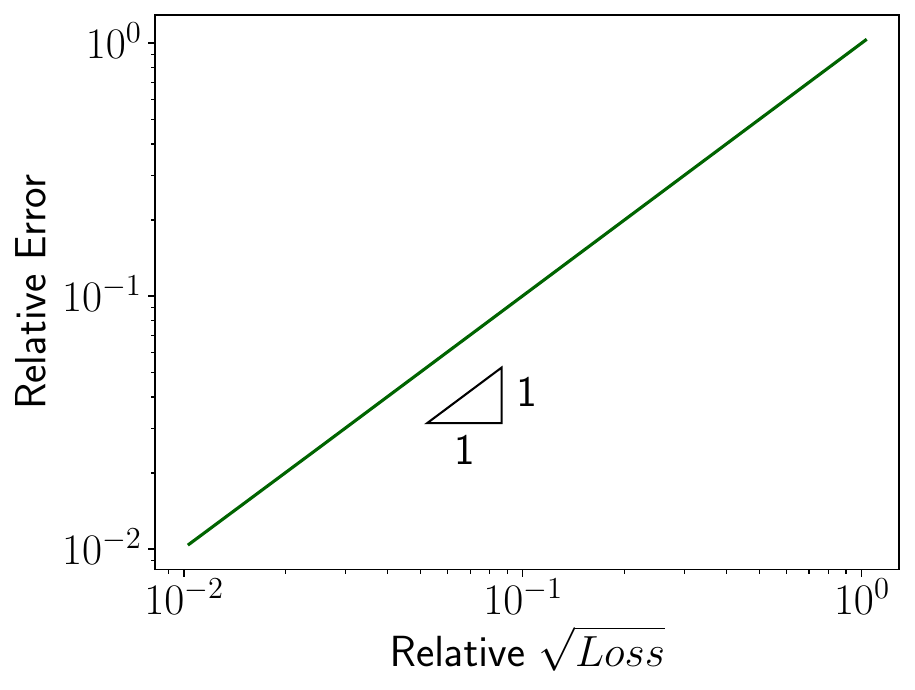}
         \caption{Error-Loss comparison}
         \label{fig:smooth_FEM_c}
     \end{subfigure}
     \hfill
    \caption{RVPINNs approximation of the smooth diffusion problem; strong BCs imposition, $100$ FE test functions.}
    \label{fig:smooth_FEM_strong_100}
\end{figure}

\begin{figure}
    \centering        
     \begin{subfigure}[b]{0.3\textwidth}
         \centering
         \includegraphics[width=\textwidth]{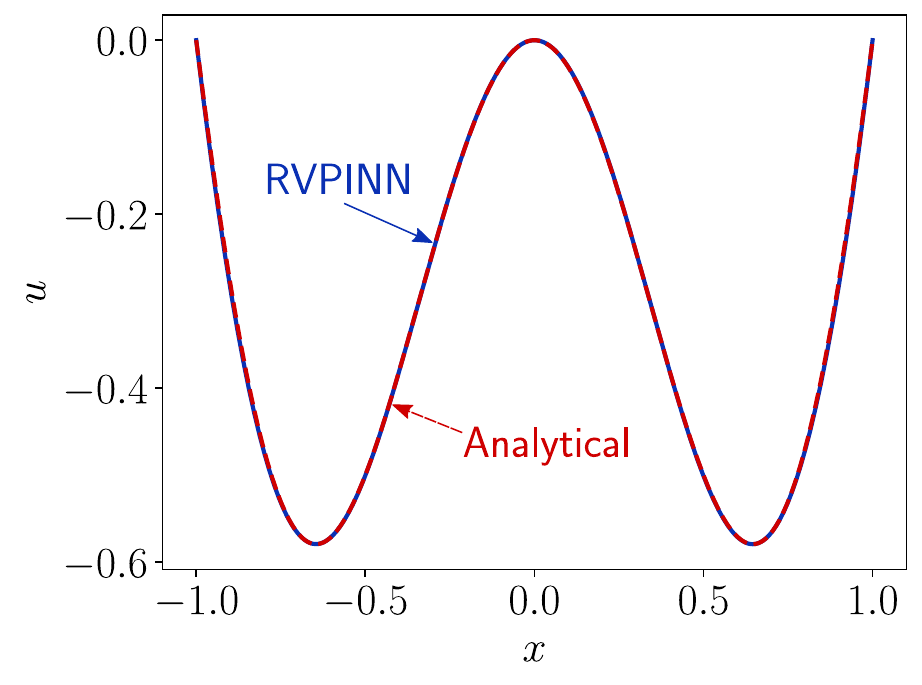}
         \caption{DNN best approximation}
     \end{subfigure}
     \hfill
     \begin{subfigure}[b]{0.3\textwidth}
         \centering
         \includegraphics[width=\textwidth]{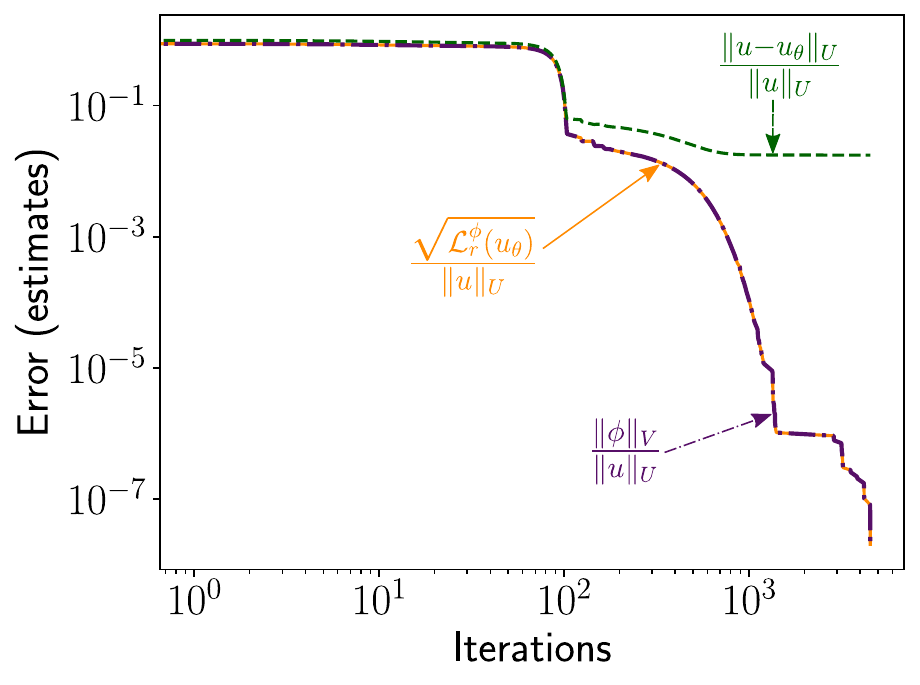}
         \caption{Estimates vs Iterations}
         \label{fig:smooth_FEM_strong_5_b}
     \end{subfigure}
     \hfill
     \begin{subfigure}[b]{0.3\textwidth}
         \centering
         \includegraphics[width=\textwidth]{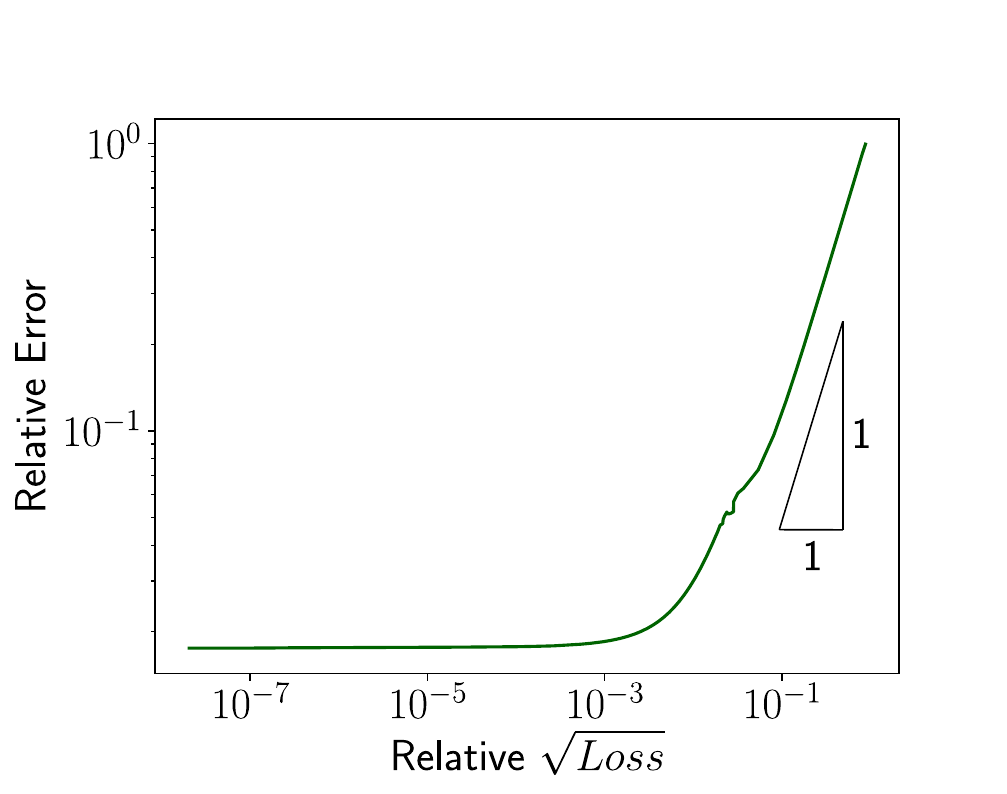}
         \caption{Error-Loss comparison}
     \end{subfigure}
     \hfill
    \caption{RVPINNs approximation of the 1D smooth diffusion problem; strong BCs imposition, $5$ FE test functions.}
    \label{fig:smooth_FEM_strong_5}
\end{figure}

\noindent
To conclude this section, we show the results obtained for a 2D pure diffusion example, set in the domain $\Omega=(0,1)^2$ analogously so the analytical solution coincides with
\begin{equation}
    u(x,y) = \sin(\pi x) \sin(\pi y).
\end{equation}
Figure~\ref{fig:2Dsmooth_spectral_strong_50} 
displays the results when we consider the spectral loss functional~\eqref{eq:spectral_2D} during training with $M=N=20$ and strong imposition of the BCs. 
Similarly,  Figure~\ref{fig:2Dsmooth_FEM_strong_50} displays the results when considering a $100\times100$-element mesh with piece-wise linear FEM test functions. We observe a similar behavior, up to some oscillation due to the inexact integration,  to the 1D case, showing in particular that the strategy is dimension-independent, as expected.\\

\noindent
To showcase the behavior of RVPINNs in extreme scenarios, in the following, we will consider 1D challenging problems for the particular setting described in Section~\ref{sec:discrete_setting}.
\begin{figure}
    \centering        
     \hfill
     \begin{subfigure}[b]{0.3\textwidth}
         \centering
         \includegraphics[width=\textwidth]{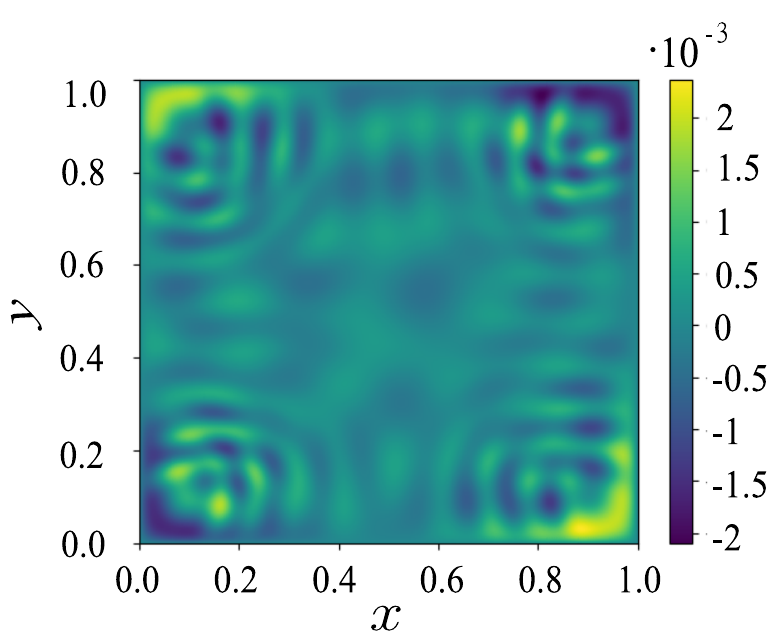}
         \caption{$u-\text{(best) }u_\theta$}
         \label{fig:2Dsmooth_spectral_b}
     \end{subfigure}
     \hfill
     \begin{subfigure}[b]{0.3\textwidth}
         \centering
         \includegraphics[width=\textwidth]{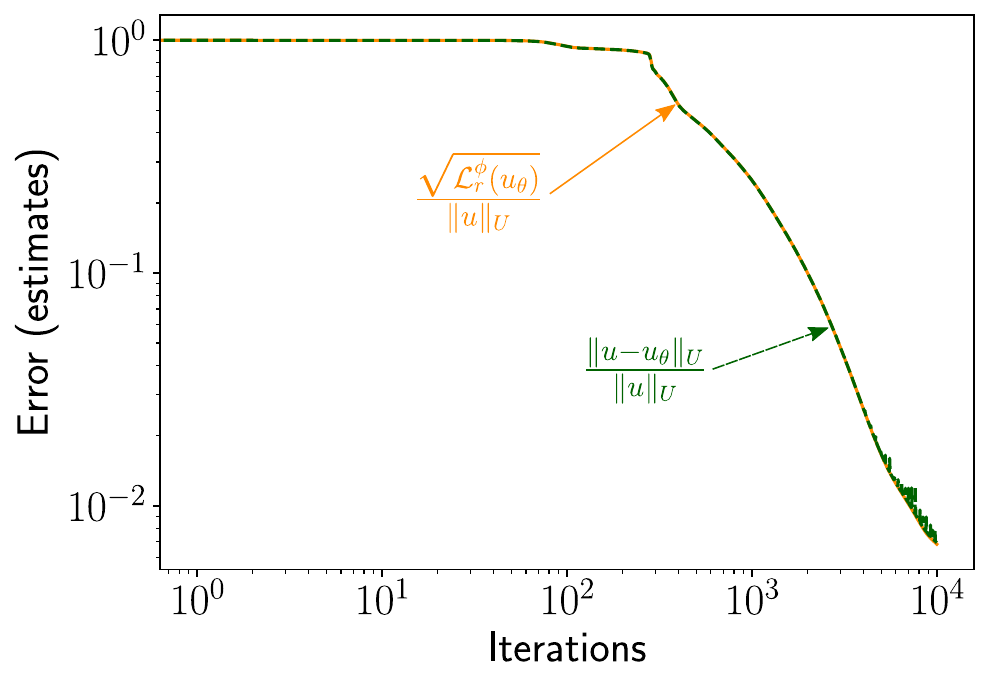}
          \caption{Estimates vs Iterations}
         \label{fig:2Dsmooth_spectral_c}
     \end{subfigure}
     \hfill
     \begin{subfigure}[b]{0.3\textwidth}
         \centering
         \includegraphics[width=\textwidth]{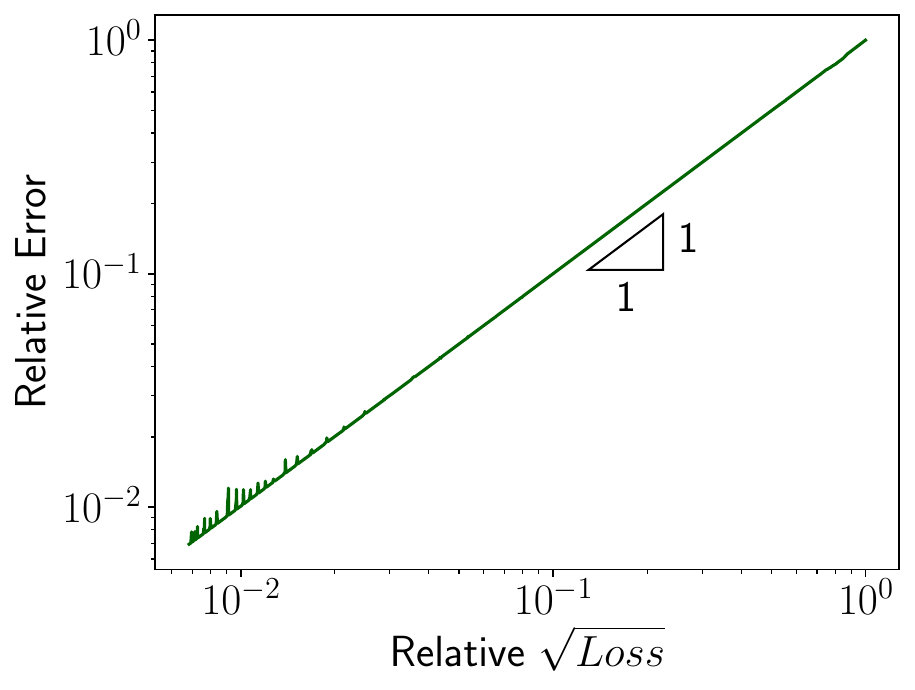}
         \caption{Error-Loss comparison}
         \label{fig:2Dsmooth_spectral_d}
     \end{subfigure}
     \hfill
    \caption{RVPINNs approximation of the 2D smooth problem; strong BCs imposition, $20 \times 20$ spectral test functions.}
    \label{fig:2Dsmooth_spectral_strong_50}
\end{figure}
\begin{figure}
    \centering        
     \hfill
     \begin{subfigure}[b]{0.3\textwidth}
         \centering
         \includegraphics[width=\textwidth]{figures/2D/2D_FEM_difference.png}
         \caption{$u-\text{(best) }u_\theta$}
         \label{fig:2Dsmooth_FEM_b}
     \end{subfigure}
     \hfill
     \begin{subfigure}[b]{0.3\textwidth}
         \centering
         \includegraphics[width=\textwidth]{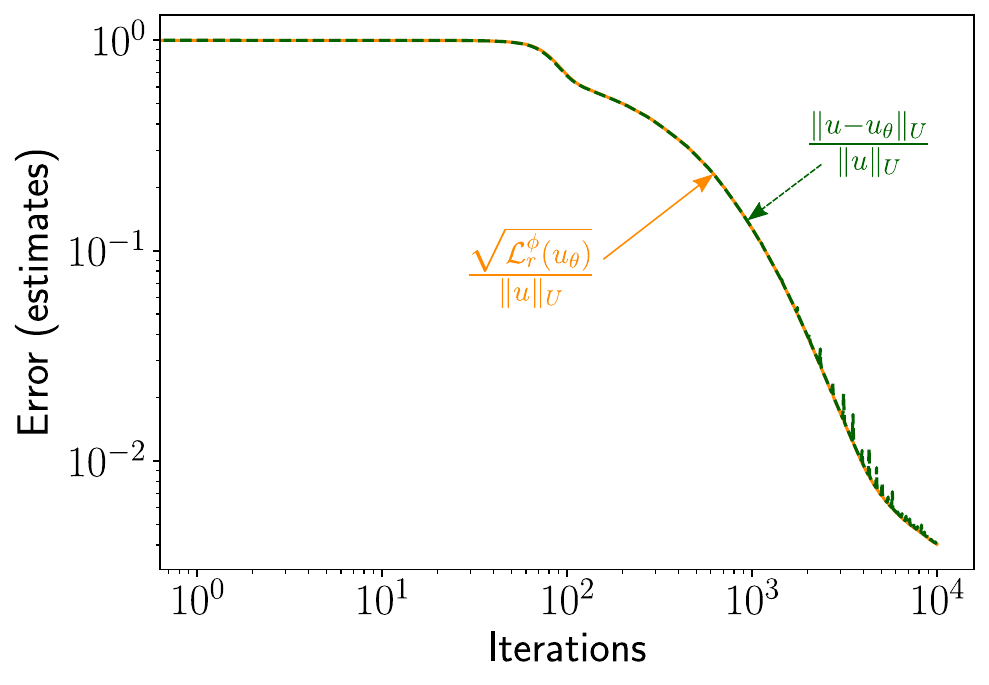}
          \caption{Estimates vs Iterations}
         \label{fig:2Dsmooth_FEM_c}
     \end{subfigure}
     \hfill
     \begin{subfigure}[b]{0.3\textwidth}
         \centering
         \includegraphics[width=\textwidth]{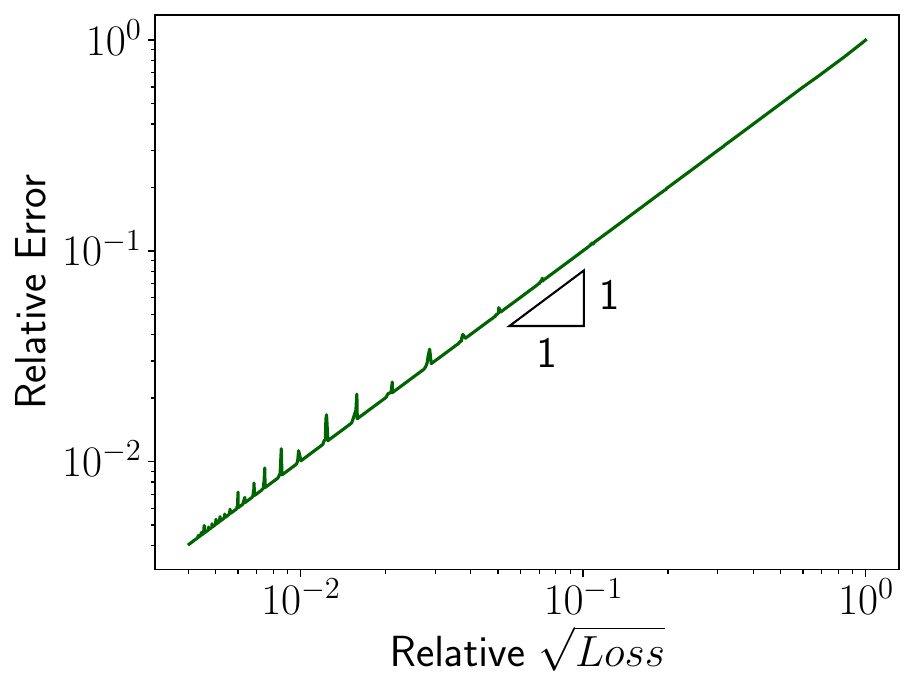}
         \caption{Error-Loss comparison}
         \label{fig:2Dsmooth_FEM_d}
     \end{subfigure}
     \hfill
    \caption{RVPINNs approximation of the 2D smooth problem; strong BCs imposition, $100 \times 100$ elements mesh with piece-wise linear FEM test functions.}
    \label{fig:2Dsmooth_FEM_strong_50}
\end{figure}

\subsection{Delta source problem}\label{sec:delta}
{We consider again the pure 1D diffusion problem (i.e., eq.~\eqref{eq:weakDAR} with $\varepsilon = 1$ and $\beta = 0$), defined over $\Omega = (-1,1)$ and subject to homogeneous Dirichlet BCs, but now with a Dirac delta source $\delta_{1/2} \in H^{-1}(\Omega)$, that is explicitly defined through the action:
\begin{equation}
l(v) := \langle \delta_{1/2}\, , \, v \rangle := v(1/2), \quad \forall \, v \in H_0^1(\Omega).
\end{equation}
The analytical solution to this problem is explicitly given as: 
\begin{equation*}
u(x) = \left\{ \begin{array}{r@{\;}r}
\dfrac{1}{4}(x+1), & \quad   -1\leq x\leq \dfrac{1}{2}, \smallskip\\
\dfrac{3}{4}(1-x), & \quad  \dfrac{1}{2} < x\leq 1,
\end{array}
\right.
\end{equation*}
which is only a $C^0(\Omega)$ solution, while the activation function $\tanh$ is smooth. Therefore, a major dominance of the oscillation term is expected, even when considering a sufficiently large number of test functions, as appreciated in Figure~\ref{fig:delta_FEM_strong_100} where we have considered $100$ FE test functions and strong imposition of the BCs. Figure~\ref{fig:delta_FEM_constrained_100} shows similar results for the same configuration but with the {weak} imposition of the BCs. In both cases, a good representation of the analytical solution is obtained despite being approximated with a smooth function. Additionally, figures~\ref{fig:delta_FEM_strong_100} and ~\ref{fig:delta_FEM_constrained_100} show that the loss functional defines a robust estimation of the error in the pre-asymptotic regime (before the dominance of the oscillation term). A similar behavior is obtained when considering the loss functional with spectral test functions-- omitted here for brevity.
\begin{figure}
    \centering        
     \begin{subfigure}[b]{0.3\textwidth}
         \centering
         \includegraphics[width=\textwidth]{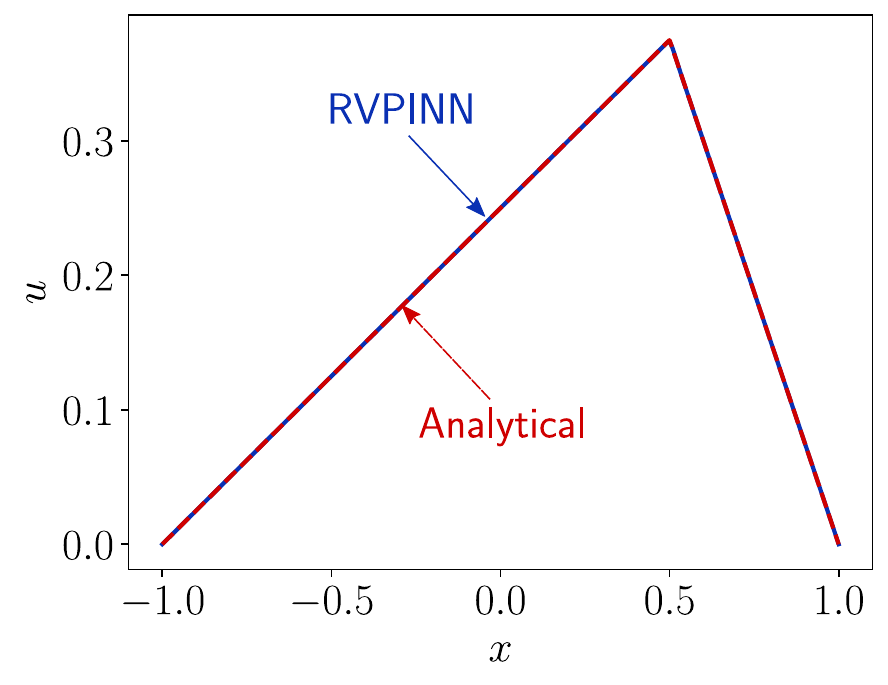}
         \caption{DNN best approximation}
     \end{subfigure}
     \hfill
     \begin{subfigure}[b]{0.3\textwidth}
         \centering
         \includegraphics[width=\textwidth]{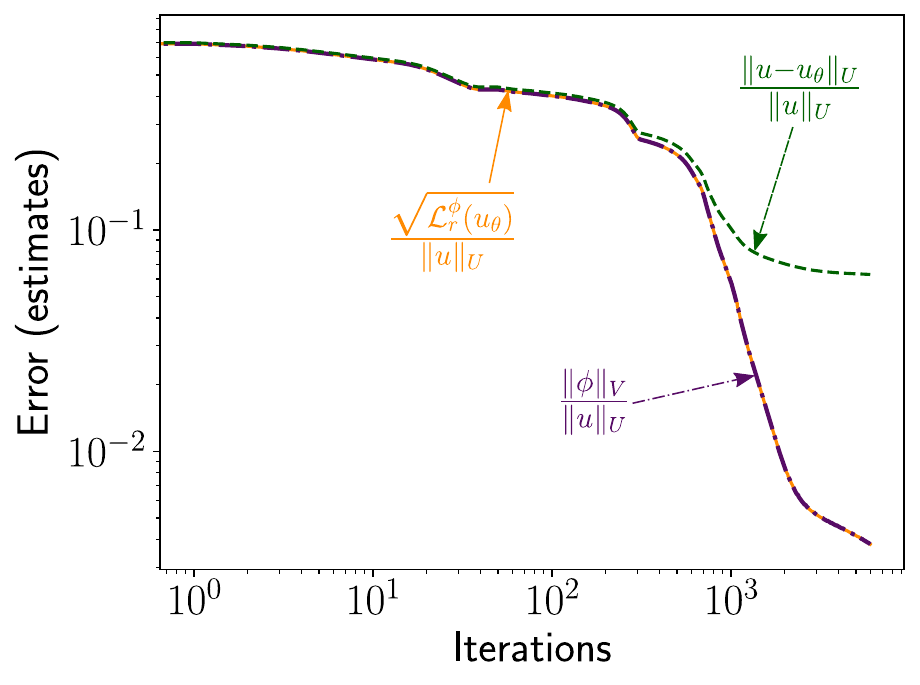}
         \caption{Estimates vs Iterations}
     \end{subfigure}
     \hfill
     \begin{subfigure}[b]{0.3\textwidth}
         \centering
         \includegraphics[width=\textwidth]{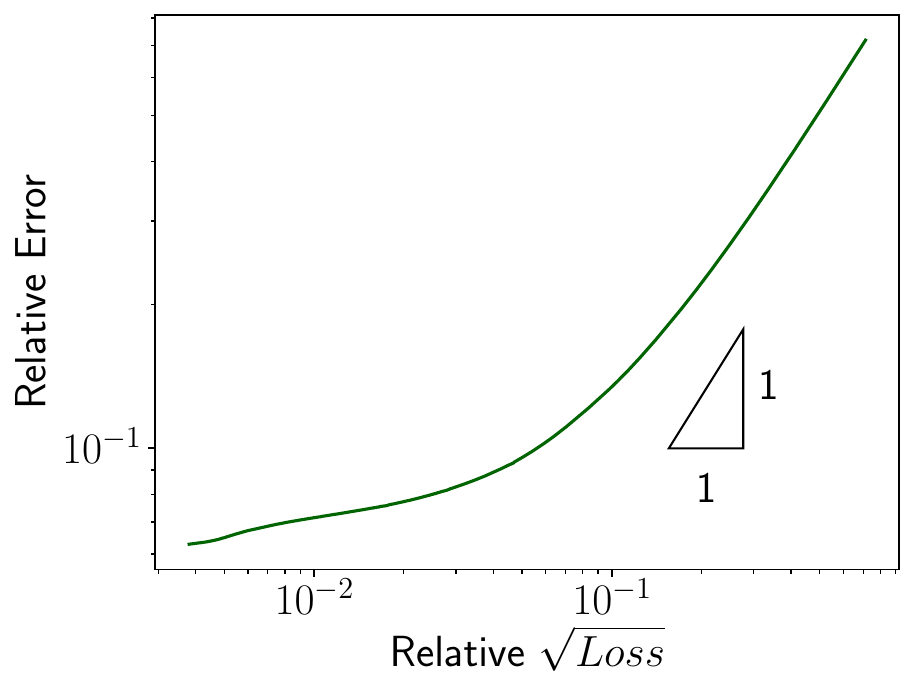}
         \caption{Error-Loss comparison}
     \end{subfigure}
     \hfill
    \caption{RVPINNs approximation of the delta source diffusion problem; strong BCs imposition, $100$ FE test functions.} 
    \label{fig:delta_FEM_strong_100}
\end{figure}

\begin{figure}
    \centering        
     \begin{subfigure}[b]{0.3\textwidth}
         \centering
         \includegraphics[width=\textwidth]{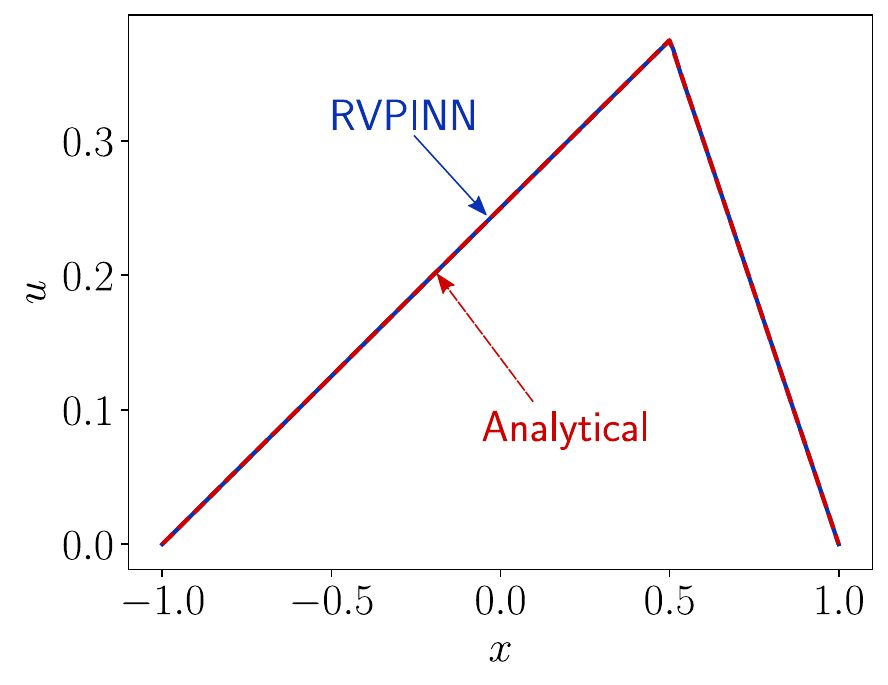}
         \caption{DNN best approximation}
     \end{subfigure}
     \hfill
     \begin{subfigure}[b]{0.3\textwidth}
         \centering
         \includegraphics[width=\textwidth]{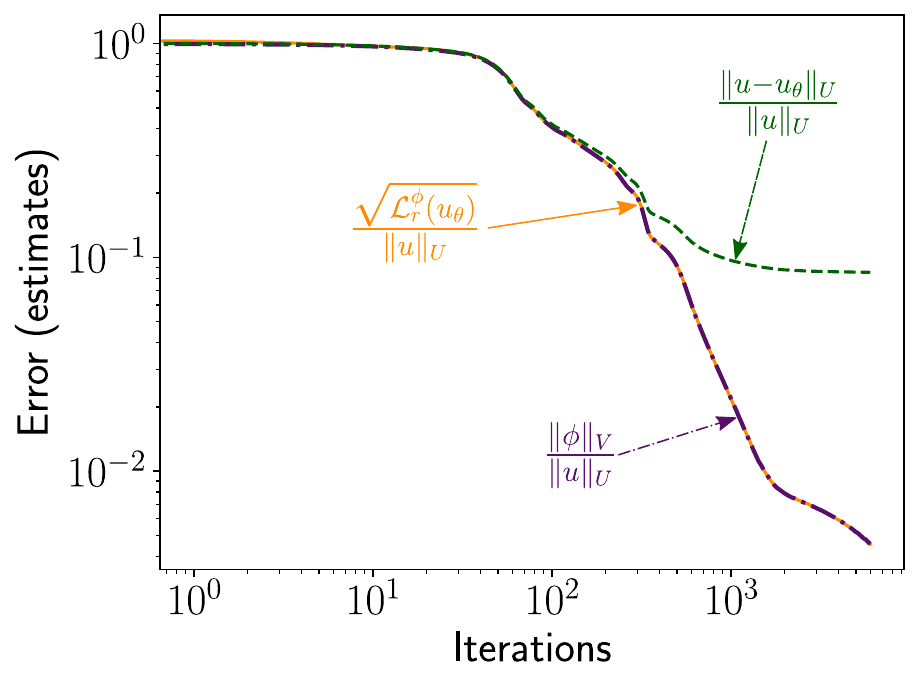}
         \caption{Estimates vs Iterations}
     \end{subfigure}
     \hfill
     \begin{subfigure}[b]{0.3\textwidth}
         \centering
         \includegraphics[width=\textwidth]{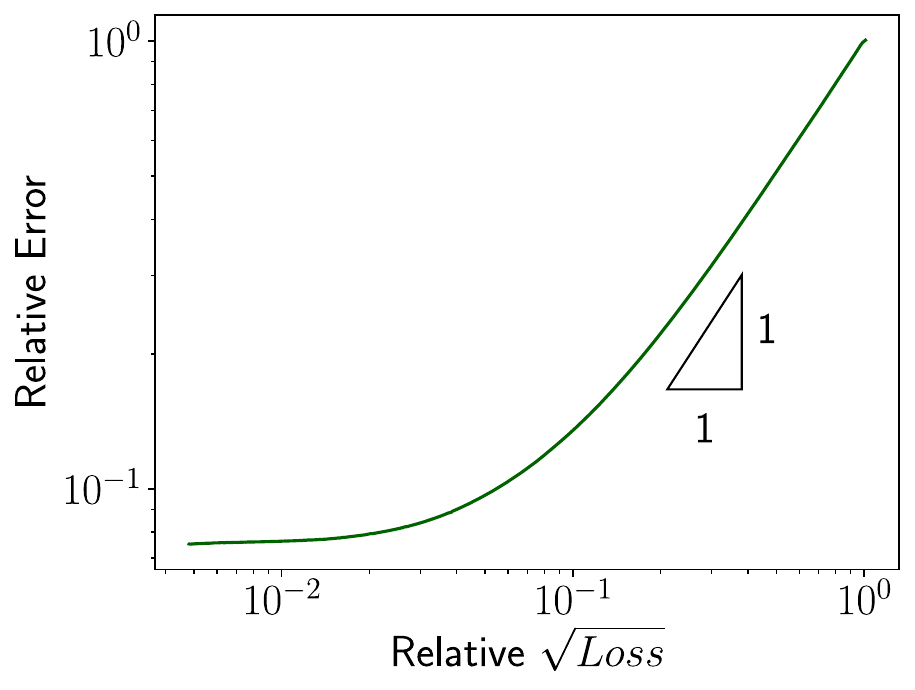}
         \caption{Error-Loss comparison}
     \end{subfigure}
     \hfill
    \caption{RVPINNs approximation of the delta diffusion problem; {weak} BCs imposition, $100$ FE test functions.}
    \label{fig:delta_FEM_constrained_100}
\end{figure}
\subsection{Advection-dominated diffusion problem}
As our last example, we consider problem ~\eqref{eq:weakDAR} in an advection-dominant regime, with the spirit of testing the performance and robustness of the estimates for RVPINNs in near unstable scenarios. For this, we set $\beta=1$, the source term as $f=1$, and consider small values of $\varepsilon$. For a given $\varepsilon$, the analytical solution of ~\eqref{eq:weakDAR} in this case is given by:
\begin{eqnarray*}
    u(x) = \dfrac{2(1-e^{\frac{x-1}{\varepsilon}})}{1-e^{-\frac{2}{\varepsilon}}} + x-1,
\end{eqnarray*}
exhibiting a strong gradient near the right boundary when $\varepsilon \rightarrow 0$ (cf. figures~\ref{fig:ad_strong_01} and \ref{fig:ad_strong_005}). We restrict the discussion to spectral test functions {(figures ~\ref{fig:da_spectral_strong_01}--~\ref{fig:da_spectral_constrained_005})}, as we observe similar behaviors when considering FE test functions.\\
    Figure~\ref{fig:da_spectral_strong_01} shows the results obtained for $\varepsilon = 0.1$, considering $50$ test functions and strong imposition of the BCs. Contrary to the pure diffusive case, Figure~\ref{fig:ad_strong_01_b} reflects that the lines representing the error and the square root of the loss functional (respectively, the norm of the discrete Riesz representative {with respect to the norm of the test space $V\times\mathbb{R}^2$}) do not overlap. This behavior is expected as, in this case, the continuity constant $\mu$ is not equal to one. {Indeed, since the Poincar\'e constant is well-known in 1D, we can explicitly calculate it: $\mu = 1+2/(\varepsilon \pi)$. It deteriorates when $\varepsilon \rightarrow 0$, {i.e., the estimations for this particular setting are not robust in terms of the physical parameter $\varepsilon$}. Nevertheless, as we can see in Figure~\ref{fig:ad_strong_01_c}, our theoretical a posteriori error bounds ~\eqref{eq:aposteriori_lower} and \eqref{eq:aposteriori_reliability_2} are satisfied.} We also observe a decay in the expected optimal correlation between the error and the square root of the loss starting near iteration $3000$, apparently due to stability issues, {that is in agreement with the oscilation term appearing in the a posteriori error estimate \eqref{eq:aposteriori_reliability_2}}. This suboptimality is improved when considering the {weak} imposition of the boundary conditions, as shown in Figure~\ref{fig:da_spectral_constrained_01}. This {weak} imposition of BCs could thus serve as a stabilization technique. This is further confirmed in Figure~\ref{fig:da_spectral_constrained_005}, where we show the results obtained for $\varepsilon=0.005$, with $200$ test basis functions and {weak} BCs. When considering instead the strong imposition of the BCs and the same number of test functions, we do not observe convergence within the first $6000$ iterations--omitted here for brevity. We observe that when considering the {weak} BCs, the loss functional exhibits some differences with the norm of the Riesz representative {in $V$, as we intentionally plot the incomplete estimation to highlight the contribution of weak BC imposition in $\mathbb{R}^2$ of the full Riesz representative (cf. Eq.~\eqref{eq:Gram_full_weak}). Nevertheless, in this case, we still observe a good correlation between the loss and the true error, as expected.}
    
    {To guarantee good estimations in RVPINNs, we need adequate variational formulations and norms for the trial and test spaces, so the resulting inf-sup and boundedness constants are independent of the physical parameters. To exemplify this with a numerical example, consider the following (weighted) alternative variational formulation (with $\Omega:=(-1,1)$)
    \begin{equation*}
        \text{Find } u \in U:=H_0^1(\Omega) : (\varepsilon \, e^{-1/\varepsilon} u^\prime\, , \, v^\prime)_0 = (e^{-1/\varepsilon} f\, , \, v)_0, \quad \forall \, v \in V:=U,
    \end{equation*}
    where the last can be deduced from~\eqref{eq:strongDAR}, by first multiplying by $e^{-1/\varepsilon} v$ (assuming enough regularity for the solution $u$), and then performing integration by parts for the term involving the diffusion term. This alternative variational formulation is bounded and coercive, with constants $\mu=\alpha=1$, with respect to the norms:
    \begin{equation*}
        \|\cdot\|_U^2 := \|\cdot\|_V^2:= (\varepsilon \, e^{-1/\varepsilon} \cdot^\prime\, , \, \cdot^\prime)_0,
    \end{equation*}
    expecting a perfect correlation of the involved quantities as in the pure diffusive case. 
    Indeed, Figure~\ref{fig:da_FEM_weighted} shows the numerical results obtained for $\varepsilon = 0.5$, considering 200 piece-wise linear functions, confirming the robustness of RVPINNs. 
    }
}
\begin{figure}
    \centering        
     \begin{subfigure}[b]{0.3\textwidth}
         \centering
         \includegraphics[width=\textwidth]{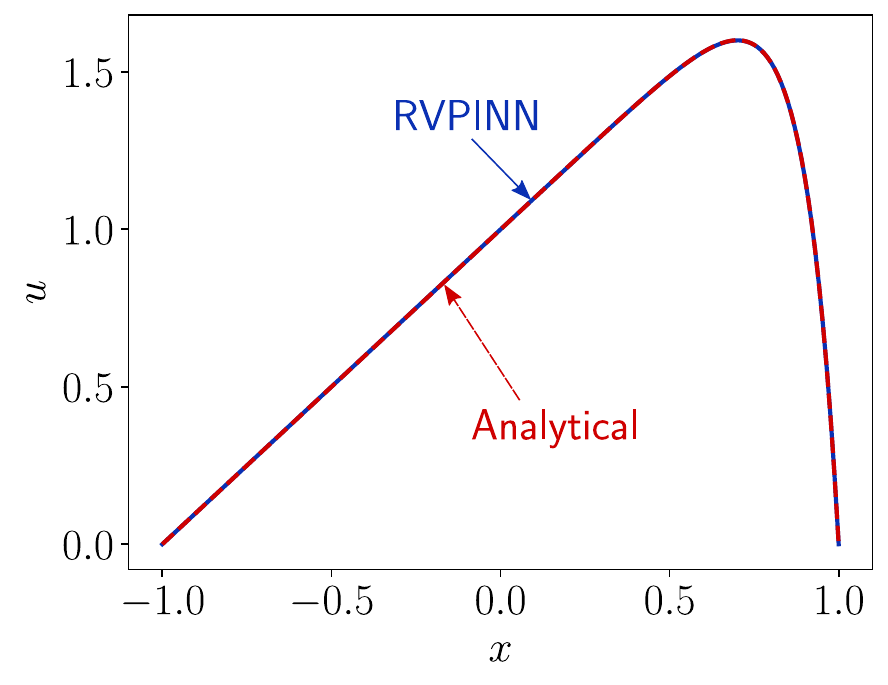}
         \caption{DNN best approximation}
         \label{fig:ad_strong_01}
     \end{subfigure}
     \hfill
     \begin{subfigure}[b]{0.3\textwidth}
         \centering
         \includegraphics[width=\textwidth]{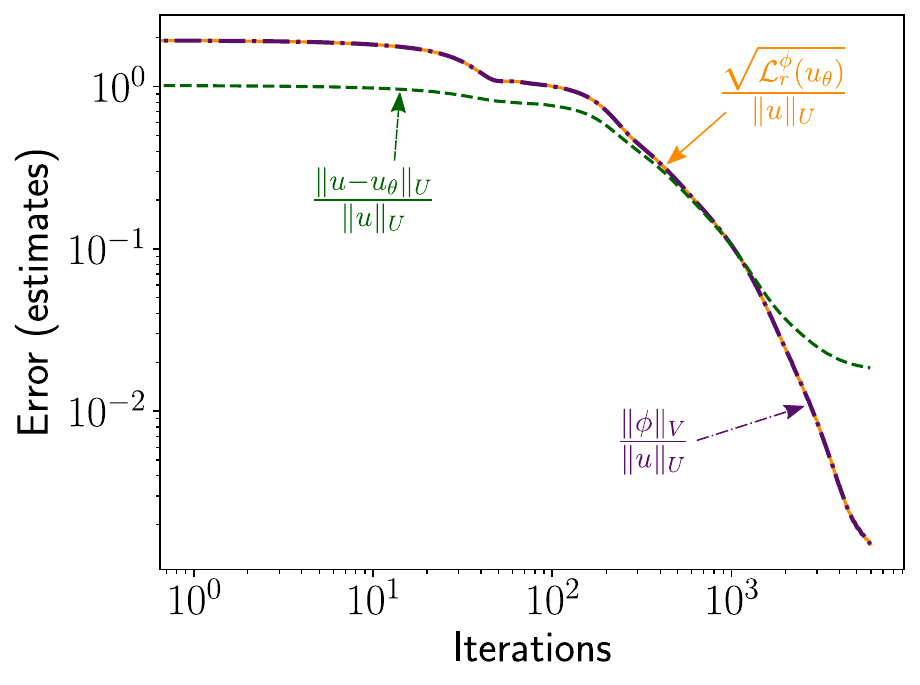}
         \caption{Estimates vs Iterations}
         \label{fig:ad_strong_01_b}
     \end{subfigure}
     \hfill
     \begin{subfigure}[b]{0.3\textwidth}
         \centering
         \includegraphics[width=\textwidth]{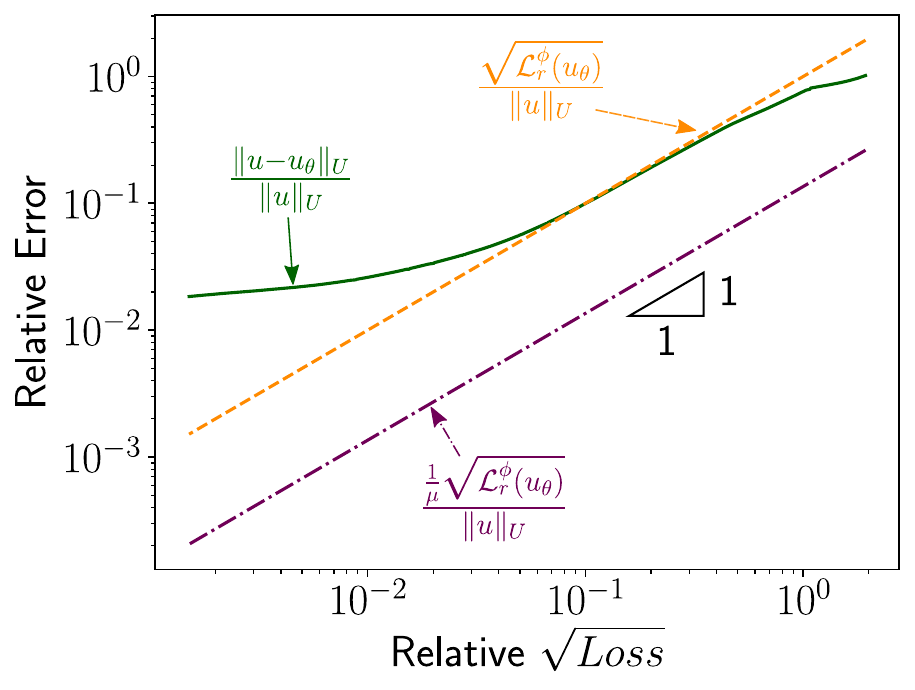}
         \caption{Error-Loss comparison}
         \label{fig:ad_strong_01_c}
     \end{subfigure}
     \hfill
    \caption{RVPINNs approximation of the diffusion-advection problem with $\varepsilon = 0.1$; strong BCs imposition,  $50$ spectral test functions. }
    \label{fig:da_spectral_strong_01}
\end{figure}

\begin{figure}
    \centering        
     \begin{subfigure}[b]{0.3\textwidth}
         \centering
         \includegraphics[width=\textwidth]{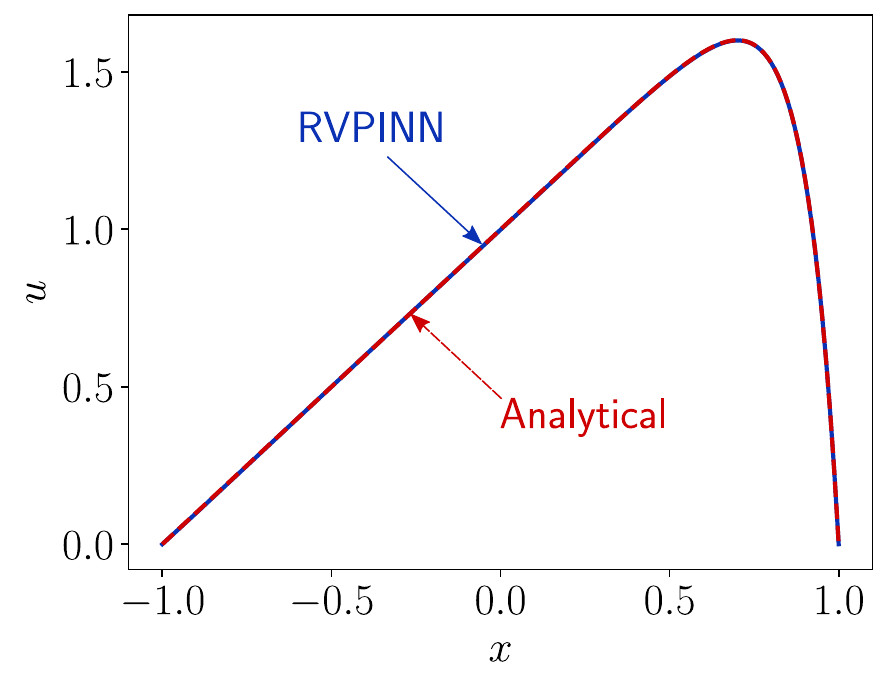}
         \caption{DNN best approximation}
     \end{subfigure}
     \hfill
     \begin{subfigure}[b]{0.3\textwidth}
         \centering
         \includegraphics[width=\textwidth]{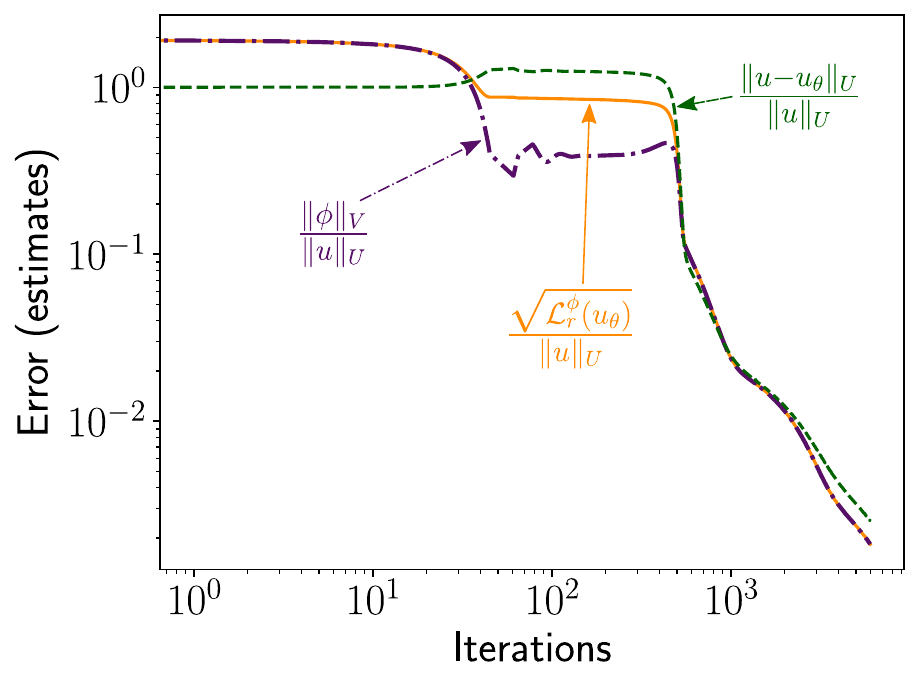}
         \caption{Estimates vs Iterations}
     \end{subfigure}
     \hfill
     \begin{subfigure}[b]{0.3\textwidth}
         \centering
         \includegraphics[width=\textwidth]{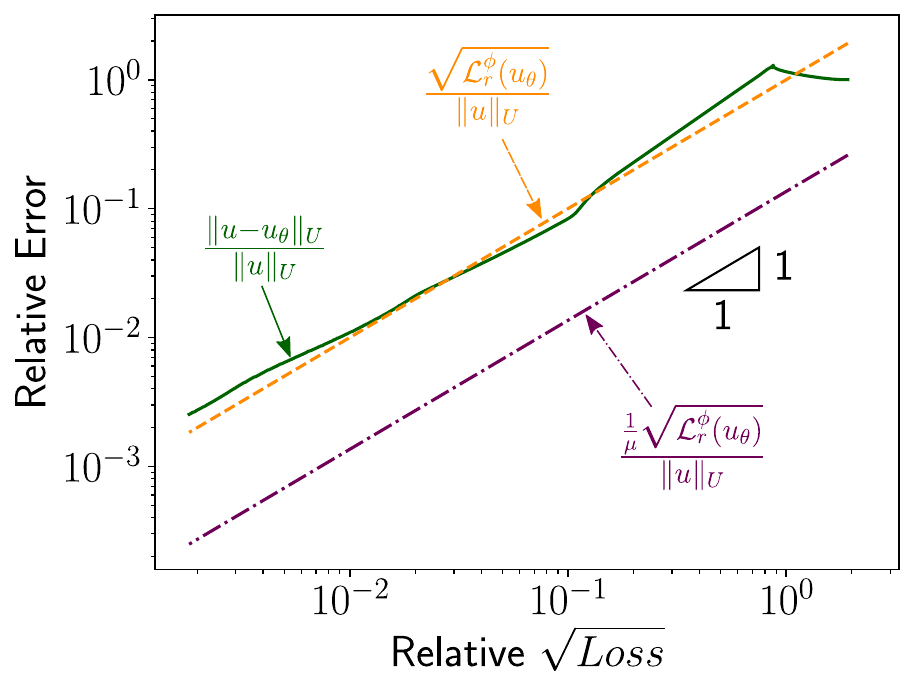}
         \caption{Error-Loss comparison}
     \end{subfigure}
     \hfill
    \caption{RVPINNs approximation of the diffusion-advection problem with $\varepsilon = 0.1$; {weak} BCs imposition, $50$ spectral test functions. }
    \label{fig:da_spectral_constrained_01}
\end{figure}

\begin{figure}
    \centering        
     \begin{subfigure}[b]{0.3\textwidth}
         \centering
         \includegraphics[width=\textwidth]{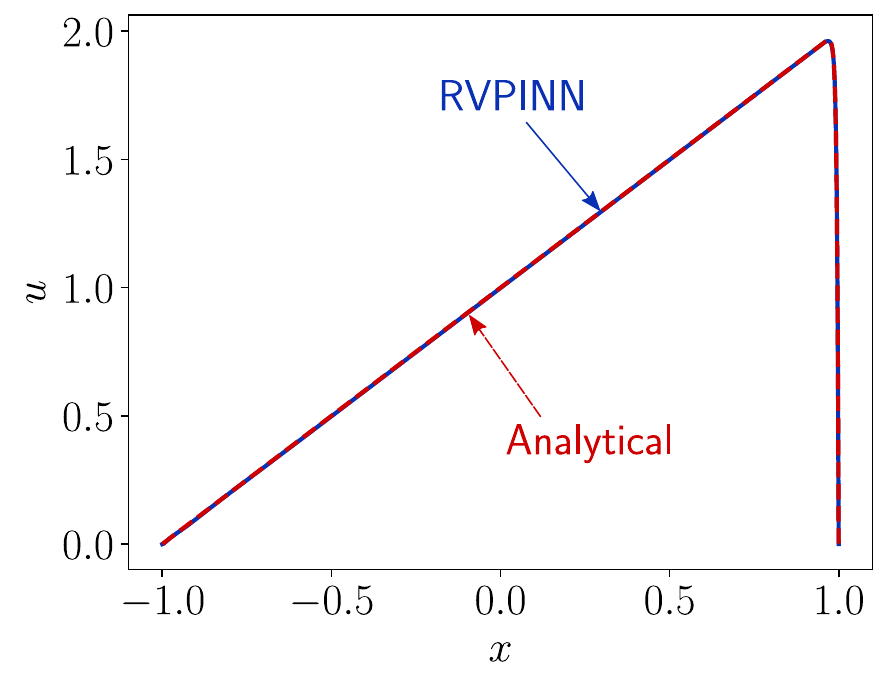}
         \caption{DNN best approximation}
         \label{fig:ad_strong_005}
     \end{subfigure}
     \hfill
     \begin{subfigure}[b]{0.3\textwidth}
         \centering
         \includegraphics[width=\textwidth]{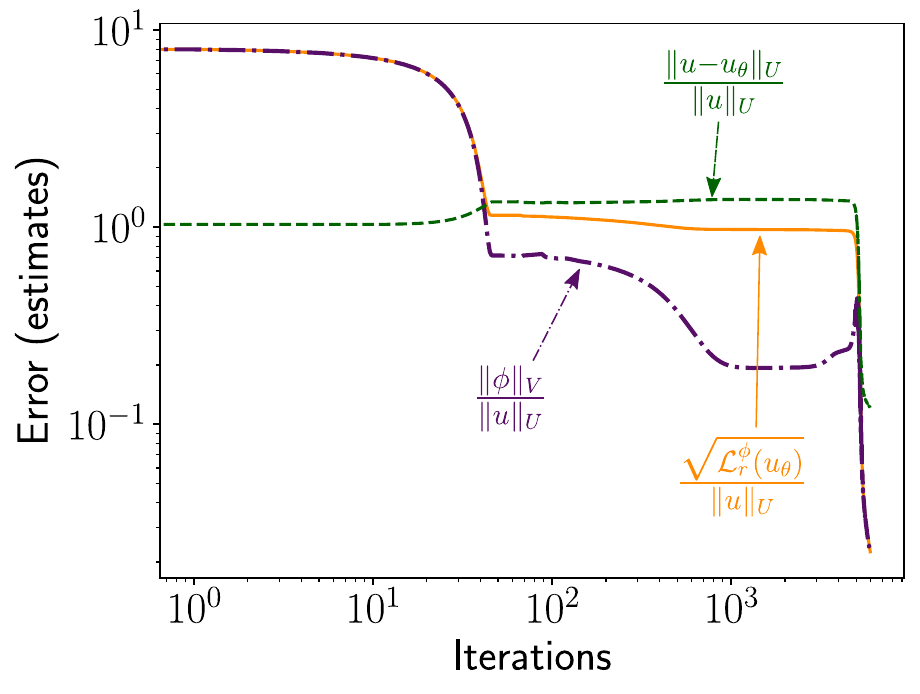}
         \caption{Estimates vs Iterations}
     \end{subfigure}
     \hfill
     \begin{subfigure}[b]{0.3\textwidth}
         \centering
         \includegraphics[width=\textwidth]{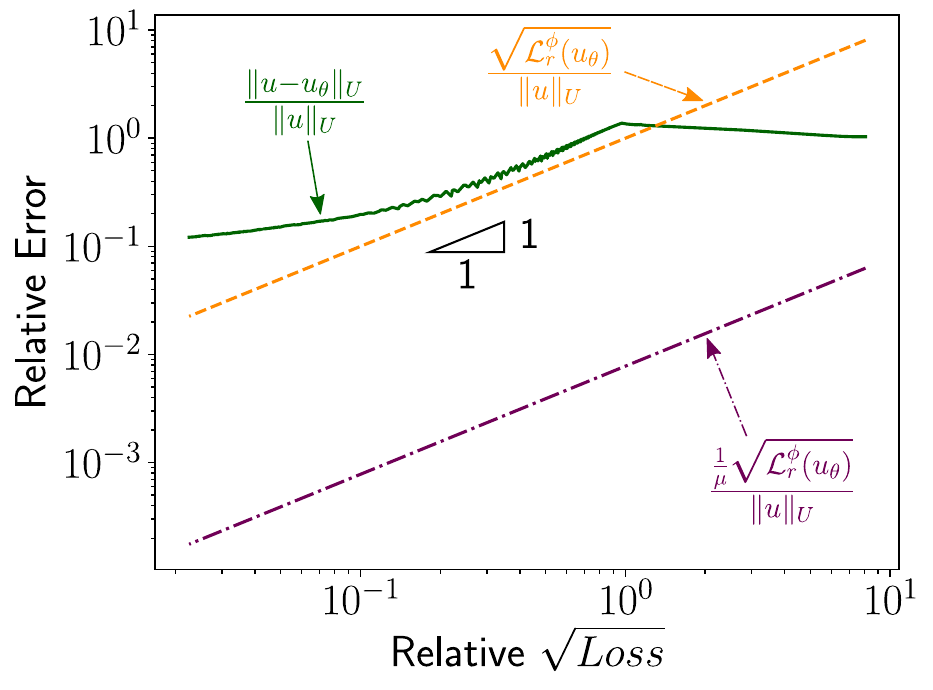}
         \caption{Error-Loss comparison}
     \end{subfigure}
     \hfill
    \caption{RVPINNs approximation of the diffusion-advection problem with $\varepsilon = 0.005$; {weak} BCs imposition, $200$ spectral test functions. }
    \label{fig:da_spectral_constrained_005}
\end{figure}
\begin{figure}
    \centering        
     \begin{subfigure}[b]{0.3\textwidth}
         \centering
         \includegraphics[width=\textwidth]{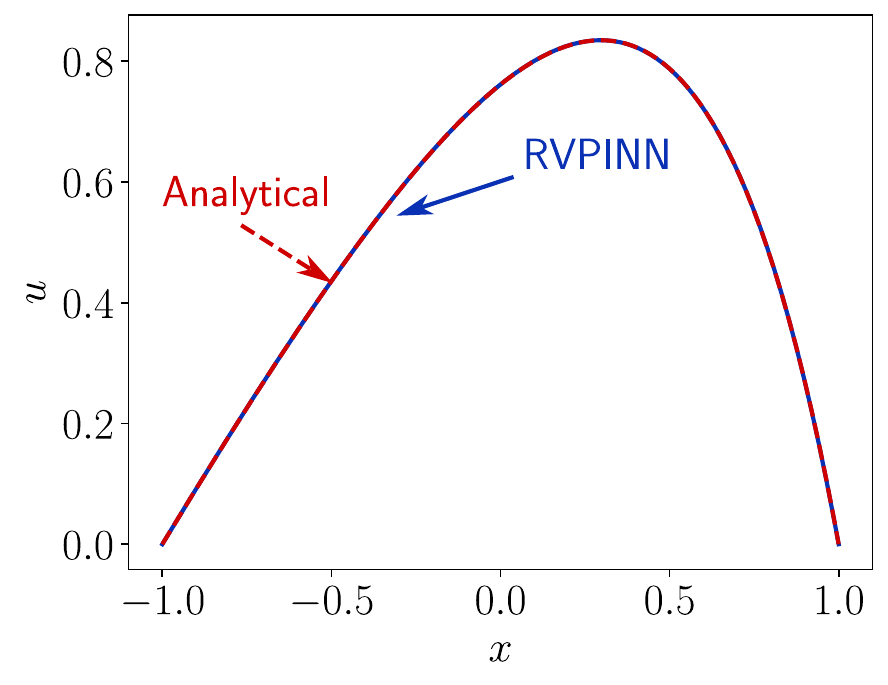}
         \caption{DNN best approximation}
         \label{fig:ad_FEM_ws}
     \end{subfigure}
     \hfill
     \begin{subfigure}[b]{0.3\textwidth}
         \centering
         \includegraphics[width=\textwidth]{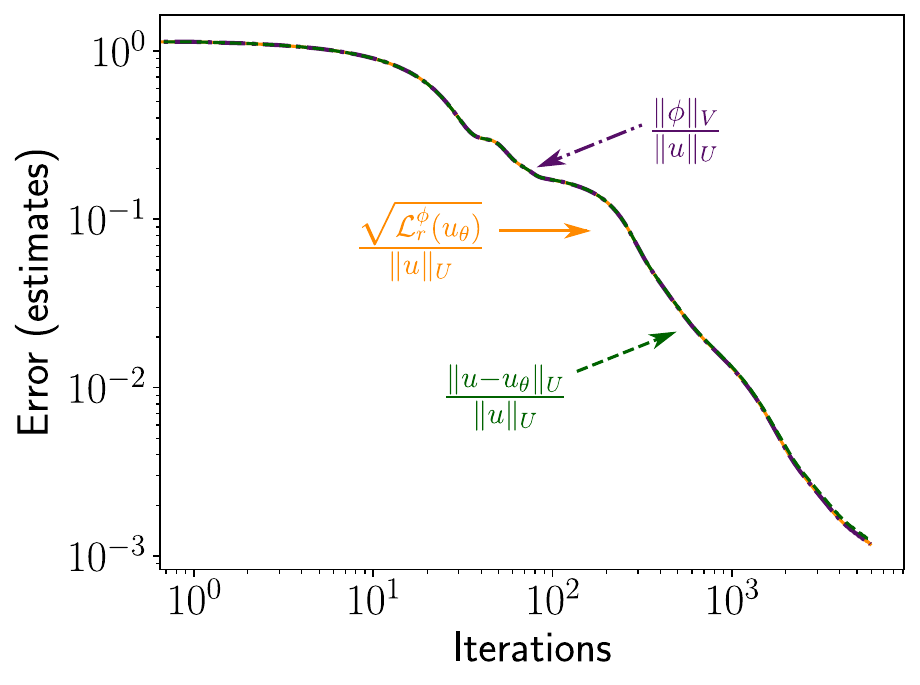}
         \caption{Estimates vs Iterations}
     \end{subfigure}
     \hfill
     \begin{subfigure}[b]{0.3\textwidth}
         \centering
         \includegraphics[width=\textwidth]{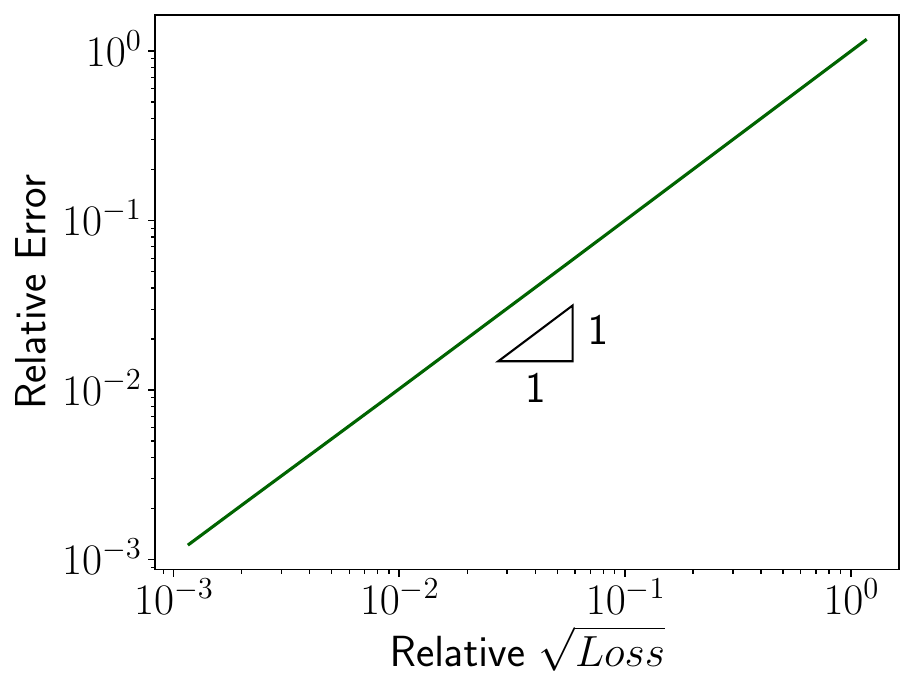}
         \caption{Error-Loss comparison}
     \end{subfigure}
     \hfill
    \caption{RVPINNs approximation of the diffusion-advection problem with $\varepsilon = 0.5$; weighted variational formulation, {strong} BCs imposition, $200$ linear FEM test functions.}
    \label{fig:da_FEM_weighted}
\end{figure}
\section{Conclusions}\label{Sec:Conc}
In this article, we provide a general mathematical framework to construct robust loss functionals based on VPINNs. For that, we first generalize the definition of the loss functional in VPINNs to the choice of a single test function. Then, following a minimum residual principle, we {define} such test function as the Riesz representative of the weak residual over a given discrete test space. We expand the Riesz representation over the discrete test space so the loss functional includes the inversion of the Gram matrix corresponding to the inner product. We prove that the true error in the energy norm is equivalent to the test norm of the residual error estimator. To prove the robustness of the method, we need to {introduce} a test norm that induces the inf-sup stability at the continuous level. We numerically show the robustness of our a posteriori error estimator and also that our methodology is insensitive to the choice of the basis functions in the selected discrete test space.\\ 
Possible future research lines include: (a) the application of the framework to other families of PDEs like wave propagation, nonlinear or time-dependent problems, (b) the extension to nonconforming formulations, (c) the application of RVPINNs to parametric problems with an off-line inversion of the Gram matrix, and (d) an efficiency study of the method for a given variational formulation, NN configuration, and test space. 

\section{Acknowledgments}
The work of Sergio Rojas was done in the framework of the Chilean grant ANID FONDECYT No. 3210009.  Judit Mu\~noz-Matute has received funding from the European Union's Horizon 2020 research and innovation programme under the Marie Sklodowska-Curie individual fellowship grant agreement No. 101017984 (GEODPG). David Pardo has received funding from: the Spanish Ministry of Science and Innovation projects with references TED2021-132783B-I00, PID2019-108111RB-I00 (FEDER/AEI) and PDC2021-121093-I00 (MCIN / AEI / 10.13039/501100011033/Next Generation EU), the “BCAM Severo Ochoa” accreditation of excellence CEX2021-001142-S / MICIN / AEI / 10.13039/501100011033; the Spanish Ministry of Economic and Digital Transformation with Misiones Project IA4TES (MIA.2021.M04.008 / NextGenerationEU PRTR); and the Basque Government through the BERC 2022-2025 program, the Elkartek project BEREZ-IA (KK-2023/00012), and the Consolidated Research Group MATHMODE (IT1456-22) given by the Department of Education.
The work of Maciej Paszy\'nski and Pawe\l{} Maczuga was supported by the program ``Excellence initiative - research university" for the AGH University of Krakow. {The authors would like to thank the anonymous reviewers for their suggestions that helped to improve the presentation of our findings. Additionally, the first author would like to thank professor Kristoffer Van der Zee for his helpful discussions related with the quasi-minimizers definition.}
\bibliographystyle{siam}
\bibliography{RVPINNs}

\begin{thebibliography}{10}

\bibitem{acosta2004optimal}
{\sc G.~Acosta and R.~Dur{\'a}n}, {\em {An optimal Poincar{\'e} inequality in
  $L^1$ for convex domains}}, Proceedings of the american mathematical society,
  132 (2004), pp.~195--202.

\bibitem{ainsworth2021galerkin}
{\sc M.~Ainsworth and J.~Dong}, {\em Galerkin neural networks: A framework for
  approximating variational equations with error control}, SIAM Journal on
  Scientific Computing, 43 (2021), pp.~A2474--A2501.

\bibitem{aldirany2023multi}
{\sc Z.~Aldirany, R.~Cottereau, M.~Laforest, and S.~Prudhomme}, {\em
  Multi-level neural networks for accurate solutions of boundary-value
  problems}, arXiv preprint arXiv:2308.11503,  (2023).

\bibitem{badia2024finite}
{\sc S.~Badia, W.~Li, and A.~F. Mart{\'\i}n}, {\em Finite element interpolated
  neural networks for solving forward and inverse problems}, Computer Methods
  in Applied Mechanics and Engineering, 418 (2024), p.~116505.

\bibitem{Berrone2022Solving}
{\sc S.~Berrone, C.~Canuto, and M.~Pintore}, {\em {Solving PDEs by variational
  physics-informed neural networks: an a posteriori error analysis}}, ANNALI
  DELL'UNIVERSITA'DI FERRARA, 68 (2022), pp.~575--595.

\bibitem{berrone2022variational}
\leavevmode\vrule height 2pt depth -1.6pt width 23pt, {\em {Variational physics
  informed neural networks: the role of quadratures and test functions}},
  Journal of Scientific Computing, 92 (2022), p.~100.

\bibitem{bochev2009least}
{\sc P.~B. Bochev and M.~D. Gunzburger}, {\em Least-squares finite element
  methods}, vol.~166, Springer Science \& Business Media, 2009.

\bibitem{boffi2013mixed}
{\sc D.~Boffi, F.~Brezzi, M.~Fortin, et~al.}, {\em {Mixed finite element
  methods and applications}}, vol.~44, Springer, 2013.

\bibitem{brevis2022neural}
{\sc I.~Brevis, I.~Muga, and K.~G. van~der Zee}, {\em Neural control of
  discrete weak formulations: Galerkin, least squares \& minimal-residual
  methods with quasi-optimal weights}, Computer Methods in Applied Mechanics
  and Engineering, 402 (2022), p.~115716.

\bibitem{BREVIS2022}
{\sc I.~Brevis, I.~Muga, and K.~G. {van der Zee}}, {\em Neural control of
  discrete weak formulations: Galerkin, least squares \& minimal-residual
  methods with quasi-optimal weights}, Computer Methods in Applied Mechanics
  and Engineering, 402 (2022), p.~115716.
\newblock A Special Issue in Honor of the Lifetime Achievements of J. Tinsley
  Oden.

\bibitem{cai2021physics}
{\sc S.~Cai, Z.~Mao, Z.~Wang, M.~Yin, and G.~E. Karniadakis}, {\em
  Physics-informed neural networks ({PINNs}) for fluid mechanics: A review},
  Acta Mechanica Sinica, 37 (2021), pp.~1727--1738.

\bibitem{cai2022least}
{\sc Z.~Cai, J.~Chen, and M.~Liu}, {\em Least-squares {ReLU} neural network
  ({LSNN}) method for scalar nonlinear hyperbolic conservation law}, Applied
  Numerical Mathematics, 174 (2022), pp.~163--176.

\bibitem{cai1994first}
{\sc Z.~Cai, R.~Lazarov, T.~A. Manteuffel, and S.~F. McCormick}, {\em
  First-order system least squares for second-order partial differential
  equations: {P}art {I}}, SIAM Journal on Numerical Analysis, 31 (1994),
  pp.~1785--1799.

\bibitem{cai1997first}
{\sc Z.~Cai, T.~A. Manteuffel, and S.~F. McCormick}, {\em First-order system
  least squares for second-order partial differential equations: {P}art {II}},
  SIAM Journal on Numerical Analysis, 34 (1997), pp.~425--454.

\bibitem{calo2020adaptive}
{\sc V.~M. Calo, A.~Ern, I.~Muga, and S.~Rojas}, {\em An adaptive stabilized
  conforming finite element method via residual minimization on dual
  discontinuous {G}alerkin norms}, Computer Methods in Applied Mechanics and
  Engineering, 363 (2020), p.~112891.

\bibitem{calo2021isogeometric}
{\sc V.~M. Calo, M.~{\L}o{\'s}, Q.~Deng, I.~Muga, and M.~Paszy{\'n}ski}, {\em
  Isogeometric residual minimization method ({iGRM}) with direction splitting
  preconditioner for stationary advection-dominated diffusion problems},
  Computer Methods in Applied Mechanics and Engineering, 373 (2021), p.~113214.

\bibitem{chen2020physics}
{\sc Y.~Chen, L.~Lu, G.~E. Karniadakis, and L.~Dal~Negro}, {\em
  Physics-informed neural networks for inverse problems in nano-optics and
  metamaterials}, Optics express, 28 (2020), pp.~11618--11633.

\bibitem{cier2021nonlinear}
{\sc R.~J. Cier, S.~Rojas, and V.~M. Calo}, {\em {A nonlinear weak constraint
  enforcement method for advection-dominated diffusion problems}}, Mechanics
  Research Communications, 112 (2021), p.~103602.

\bibitem{cier2021automatically}
\leavevmode\vrule height 2pt depth -1.6pt width 23pt, {\em Automatically
  adaptive, stabilized finite element method via residual minimization for
  heterogeneous, anisotropic advection--diffusion--reaction problems}, Computer
  Methods in Applied Mechanics and Engineering, 385 (2021), p.~114027.

\bibitem{demkowicz2006babuvska}
{\sc L.~Demkowicz}, {\em {Babu{\v{s}}ka $\Longleftrightarrow$ Brezzi}}, ICES
  Report,  (2006), pp.~06--08.

\bibitem{DEMKOWICZ20101558}
{\sc L.~Demkowicz and J.~Gopalakrishnan}, {\em {A class of discontinuous
  Petrov-{G}alerkin methods. Part I: The transport equation}}, Comput. Methods
  Appl. Mech. Eng., 199 (2010), pp.~1558 -- 1572.

\bibitem{demkowicz2013robust}
{\sc L.~Demkowicz and N.~Heuer}, {\em Robust {DPG} method for
  convection-dominated diffusion problems}, SIAM Journal on Numerical Analysis,
  51 (2013), pp.~2514--2537.

\bibitem{demkowicz2014overview}
{\sc L.~F. Demkowicz and J.~Gopalakrishnan}, {\em {An overview of the
  discontinuous Petrov Galerkin method}}, Recent Developments in Discontinuous
  Galerkin Finite Element Methods for Partial Differential Equations: 2012 John
  H Barrett Memorial Lectures,  (2014), pp.~149--180.

\bibitem{di2011mathematical}
{\sc D.~A. Di~Pietro and A.~Ern}, {\em {Mathematical aspects of discontinuous
  {G}alerkin methods}}, vol.~69, Springer Science, 2012.

\bibitem{ern2016converse}
{\sc A.~Ern and J.-L. Guermond}, {\em {A converse to Fortin's Lemma in Banach
  spaces}}, Comptes Rendus Mathematique, 354 (2016), pp.~1092--1095.

\bibitem{fortin1977analysis}
{\sc M.~Fortin}, {\em {An analysis of the convergence of mixed finite element
  methods}}, RAIRO. Analyse num{\'e}rique, 11 (1977), pp.~341--354.

\bibitem{gao2022physics}
{\sc H.~Gao, M.~J. Zahr, and J.-X. Wang}, {\em {Physics-informed graph neural
  Galerkin networks: A unified framework for solving PDE-governed forward and
  inverse problems}}, Computer Methods in Applied Mechanics and Engineering,
  390 (2022), p.~114502.

\bibitem{gheisari2017survey}
{\sc M.~Gheisari, G.~Wang, and M.~Z.~A. Bhuiyan}, {\em A survey on deep
  learning in big data}, in 2017 IEEE international conference on computational
  science and engineering (CSE) and IEEE international conference on embedded
  and ubiquitous computing (EUC), vol.~2, IEEE, 2017, pp.~173--180.

\bibitem{hinton2012deep}
{\sc G.~Hinton, L.~Deng, D.~Yu, G.~E. Dahl, A.-r. Mohamed, N.~Jaitly,
  A.~Senior, V.~Vanhoucke, P.~Nguyen, T.~N. Sainath, et~al.}, {\em Deep neural
  networks for acoustic modeling in speech recognition: The shared views of
  four research groups}, IEEE Signal processing magazine, 29 (2012),
  pp.~82--97.

\bibitem{hughes1989new}
{\sc T.~J. Hughes, L.~P. Franca, and G.~M. Hulbert}, {\em A new finite element
  formulation for computational fluid dynamics: Viii. the
  galerkin/least-squares method for advective-diffusive equations}, Computer
  methods in applied mechanics and engineering, 73 (1989), pp.~173--189.

\bibitem{jiang1998least}
{\sc B.-n. Jiang}, {\em The least-squares finite element method: theory and
  applications in computational fluid dynamics and electromagnetics}, Springer
  Science \& Business Media, 1998.

\bibitem{kharazmi2019variational}
{\sc E.~Kharazmi, Z.~Zhang, and G.~E. Karniadakis}, {\em Variational
  physics-informed neural networks for solving partial differential equations},
  arXiv preprint arXiv:1912.00873,  (2019).

\bibitem{kharazmi2021hp}
\leavevmode\vrule height 2pt depth -1.6pt width 23pt, {\em hp-vpinns:
  Variational physics-informed neural networks with domain decomposition},
  Computer Methods in Applied Mechanics and Engineering, 374 (2021), p.~113547.

\bibitem{KHARAZMI2021113547}
{\sc E.~Kharazmi, Z.~Zhang, and G.~E. Karniadakis}, {\em {hp-VPINNs:
  Variational physics-informed neural networks with domain decomposition}},
  Computer Methods in Applied Mechanics and Engineering, 374 (2021), p.~113547.

\bibitem{kingma2014adam}
{\sc D.~P. Kingma and J.~Ba}, {\em {Adam: A method for stochastic
  optimization}}, arXiv preprint arXiv:1412.6980,  (2014).

\bibitem{krishnapriyan2021characterizing}
{\sc A.~Krishnapriyan, A.~Gholami, S.~Zhe, R.~Kirby, and M.~W. Mahoney}, {\em
  Characterizing possible failure modes in physics-informed neural networks},
  Advances in Neural Information Processing Systems, 34 (2021),
  pp.~26548--26560.

\bibitem{krizhevsky2017imagenet}
{\sc A.~Krizhevsky, I.~Sutskever, and G.~E. Hinton}, {\em Imagenet
  classification with deep convolutional neural networks}, Communications of
  the ACM, 60 (2017), pp.~84--90.

\bibitem{kyburg2022incompressible}
{\sc F.~E. Kyburg, S.~Rojas, and V.~M. Calo}, {\em {Incompressible flow
  modeling using an adaptive stabilized finite element method based on residual
  minimization}}, International Journal for Numerical Methods in Engineering,
  123 (2022), pp.~1717--1735.

\bibitem{los2021isogeometric}
{\sc M.~{\L}o{\'s}, I.~Muga, J.~Mu{\~n}oz-Matute, and M.~Paszy{\'n}ski}, {\em
  Isogeometric residual minimization {(iGRM)} for non-stationary {S}tokes and
  {N}avier--{S}tokes problems}, Computers \& Mathematics with Applications, 95
  (2021), pp.~200--214.

\bibitem{los2020isogeometric}
{\sc M.~{\L}o{\'s}, J.~Mu{\~n}oz-Matute, I.~Muga, and M.~Paszy{\'n}ski}, {\em
  Isogeometric residual minimization method ({iGRM}) with direction splitting
  for non-stationary advection--diffusion problems}, Computers \& Mathematics
  with Applications, 79 (2020), pp.~213--229.

\bibitem{los2021dgirm}
{\sc M.~{\L}o{\'s}, S.~Rojas, M.~Paszy{\'n}ski, I.~Muga, and V.~M. Calo}, {\em
  {DGIRM: Discontinuous Galerkin based isogeometric residual minimization for
  the Stokes problem}}, Journal of Computational Science, 50 (2021), p.~101306.

\bibitem{mao2020physics}
{\sc Z.~Mao, A.~D. Jagtap, and G.~E. Karniadakis}, {\em Physics-informed neural
  networks for high-speed flows}, Computer Methods in Applied Mechanics and
  Engineering, 360 (2020), p.~112789.

\bibitem{mishra2022estimates}
{\sc S.~Mishra and R.~Molinaro}, {\em Estimates on the generalization error of
  physics-informed neural networks for approximating a class of inverse
  problems for {PDEs}}, IMA Journal of Numerical Analysis, 42 (2022),
  pp.~981--1022.

\bibitem{munoz2021dpg}
{\sc J.~Mu{\~n}oz-Matute, D.~Pardo, and L.~Demkowicz}, {\em {A DPG-based
  time-marching scheme for linear hyperbolic problems}}, Computer Methods in
  Applied Mechanics and Engineering, 373 (2021), p.~113539.

\bibitem{nagaraj2017construction}
{\sc S.~Nagaraj, S.~Petrides, and L.~F. Demkowicz}, {\em Construction of {DPG}
  {F}ortin operators for second order problems}, Computers \& Mathematics with
  Applications, 74 (2017), pp.~1964--1980.

\bibitem{oden2017applied}
{\sc J.~T. Oden and L.~Demkowicz}, {\em Applied functional analysis}, CRC
  press, 2017.

\bibitem{payne_weinberger_1960}
{\sc L.~Payne and H.~Weinberger}, {\em An optimal {P}oincar{\'e} inequality for
  convex domains}, Arch. Ration. Mech. Anal., 5 (1960), pp.~286--292.

\bibitem{raissi2019physics}
{\sc M.~Raissi, P.~Perdikaris, and G.~E. Karniadakis}, {\em Physics-informed
  neural networks: A deep learning framework for solving forward and inverse
  problems involving nonlinear partial differential equations}, Journal of
  Computational physics, 378 (2019), pp.~686--707.

\bibitem{rasht2022physics}
{\sc M.~Rasht-Behesht, C.~Huber, K.~Shukla, and G.~E. Karniadakis}, {\em
  Physics-informed neural networks (pinns) for wave propagation and full
  waveform inversions}, Journal of Geophysical Research: Solid Earth, 127
  (2022), p.~e2021JB023120.

\bibitem{roberts2014dpg}
{\sc N.~V. Roberts, T.~Bui-Thanh, and L.~Demkowicz}, {\em {The DPG method for
  the {S}tokes problem}}, Computers \& Mathematics with Applications, 67
  (2014), pp.~966--995.

\bibitem{rojas2021goal}
{\sc S.~Rojas, D.~Pardo, P.~Behnoudfar, and V.~M. Calo}, {\em Goal-oriented
  adaptivity for a conforming residual minimization method in a dual
  discontinuous {G}alerkin norm}, Computer Methods in Applied Mechanics and
  Engineering, 377 (2021), p.~113686.

\bibitem{saad2003iterative}
{\sc Y.~Saad}, {\em Iterative methods for sparse linear systems}, SIAM, 2003.

\bibitem{taylor2023deep2}
{\sc J.~M. Taylor, M.~Bastidas, D.~Pardo, and I.~Muga}, {\em {Deep Fourier
  Residual method for solving time-harmonic Maxwell's equations}}, arXiv
  preprint arXiv:2305.09578,  (2023).

\bibitem{taylor2023deep}
{\sc J.~M. Taylor, D.~Pardo, and I.~Muga}, {\em A deep {F}ourier residual
  method for solving {PDE}s using neural networks}, Computer Methods in Applied
  Mechanics and Engineering, 405 (2023), p.~115850.

\bibitem{uriarte2023deep}
{\sc C.~Uriarte, D.~Pardo, I.~Muga, and J.~Mu{\~n}oz-Matute}, {\em {A Deep
  Double Ritz Method (D2RM) for solving Partial Differential Equations using
  Neural Networks}}, Computer Methods in Applied Mechanics and Engineering, 405
  (2023), p.~115892.

\end{thebibliography}

\end{document}